\documentclass{amsart}
\usepackage[utf8]{inputenc}
\usepackage[english]{babel}

\usepackage{mathrsfs}
\usepackage{enumitem}
\usepackage{amsmath, amsthm, amssymb, amsfonts, fullpage,graphicx}
\usepackage{hyperref, faktor}
\hypersetup{
    colorlinks,
    citecolor=red,
    filecolor=blue,
    linkcolor=blue,
    urlcolor=black
} 

\newtheorem{theorem}{Theorem}[subsection]
\newtheorem{thmx}{Theorem}
\newtheorem{corollary}{Corollary}[theorem]
\newtheorem{lemma}[theorem]{Lemma}
\newtheorem{proposition}{Proposition}[subsection]
\newtheorem{definition}{Definition}[subsection]

\newtheorem*{remark}{Remark}
\newlist{gross}{itemize}{1}
\setlist[gross,1]{label=,labelwidth=1.2in,align=parleft,itemsep=0.1\baselineskip,leftmargin=10em}

\def\NN{{\mathbb N}}

\def\RR{{\mathbb R}}
\def\TT{{\mathbb T}}

\def\SS{{\mathbb S}}
\def\VV{{\mathbb V}}
\def\EE{{\mathbb E}}

\def\0{{\mathbf 0}}
\def\1{{\mathbf 1}}

\def\Bcal{{\mathcal B}}
\def\Ccal{{\mathcal C}}

\def\Ecal{{\mathcal E}}
\def\Fcal{{\mathcal F}}

\def\Ical{{\mathcal I}}

\def\Ocal{{\mathcal O}}

\def\Scal{{\mathcal S}}

\def\Vcal{{\mathcal V}}

\def\sup{\mathrm{sup}}

\usepackage{hyperref}

\begin{document}

%Topmatter 

%One author
\title{Explicit Dynamical Systems on the Sierpi\'nski curve}
\author{Worapan Homsomboon}
\address{Oregon State University}  
\email{homsombw@oregonstate.edu, homsombw@gmail.com}
 
%\author{Thomas A. Schmidt}
%\address{Oregon State University\\Corvallis, OR 97331}
%\email{toms@math.orst.edu}
%\keywords{pseudo-Anosov, SAF invariant, flux, translation surface, Veech group}
%\subjclass[2010]{37D99,(30F60,37A25,37A45,37F99, 58F15)}
\date{Last update on 7 September 2022}

\maketitle
%%%%%%%%%%%%%%%%%%%%%%%%%%%%%%%%%%%%%%%%%%%%%%%%%%%%%%%
\begin{abstract}
We apply Boro\'nski and Oprocha's inverse limit construction of dynamical systems on the Sierpi\'nski carpet by using the initial systems of $n-$Chamanara surfaces and their $n-$baker transformations, $n \geq 2$. We show that all positive real numbers are realized as metric entropy values of dynamical systems on the carpet. We also produce a simplification of Boro\'nski and Oprocha's proof showing that dynamical systems on the carpet do not have the Bowen specification property.

\end{abstract}
%%%%%%%%%%%%%%%%%%%%%%%%%%%%%%%%%%%%%%%%%%%%%%%%%%%%%%%
\tableofcontents
%%%%%%%%%%%%%%%%%%%%%%%%%%%%%%%%%%%%%%%%%%%%%%%%%%%%%%
\section{Introduction}
A {\bf Sierpi\'nski carpet} is a plane fractal created by Wac\l aw Sierpi\'nski in 1916. It is a continuum (i.e. a nonempty, compact, connected, metrizable topological space). Furthermore, it is a universal plane curve (i.e. it contains a homeomorphic copy of any subspace of $\RR^2$ with topological dimension $1$; see say \cite{Da}). Its fractal structure also allows both theoretical exploration of the structures such as \textit{Dirichlet forms} (see \cite{Bar}), and practical application in the field of communication (WiFi antenna, see for example \cite{Anten}).

Nevertheless, the carpet does not only attract interests from topologists, but also from dynamicists. Various studies on homeomorphisms on the carpet have been done over the years. In 1991 (see \cite{Ka}), Kato showed that the Sierpi\'nski curve does not admit an expansive homeomorphism. In addition, Aarts and Oversteegen were able to show that it is possible for carpet homeomorphisms to be transitive (see \cite{AO}). A construction of a homeomorphism given by Bi\'s, Nakayama and Walczak in \cite{BNW} verified that the carpet admits a homeomorphism with positive topological entropy. 

Recently, in 2018, Jan P. Boro\'nski and Piotr Oprocha introduced in their paper \cite{BO} a new way to construct a Sierpi\'nski curve. In particular, they built an inverse limit system induced from Arnold's cat map on the 2-torus $\TT^2$. The inverse limit was shown to be homeomorphic to the Sierpi\'nski carpet. In addition, the construction produced a function on the inverse limit which was proved to be a homeomorphism. One can say that the strength of Boro\'nski and Oprocha's inverse limit construction is that it produces both the Sierpi\'nski carpet and its homeomorphism. Another important point is that their construction is based on Whyburn's topological characterization of the Sierpi\'nski carpet (see \cite{why}). Precisely, one can obtain a Sierpi\'nski curve by deleting an infinite sequence of open discs $(D_i)_{i\geq 1}$ from the $2-$sphere $\SS^2$ where
\begin{itemize}
    \item $\overline{D_i} \cap \overline{D_j} = \emptyset$ for all $i \neq j$ (i.e. their closures are pairwise disjoint),
    \item $\text{diam}(D_j) \rightarrow 0$ as $j \rightarrow \infty$, and
    \item $\cup_{i=1}^\infty \overline{D_j}$ is dense on $\SS^2$.
\end{itemize}
This characterization allows a possibility to extend the inverse limit construction: the initial system \textit{Arnold's cat map on $\TT^2$} can possibly be replaced by a dynamical system $(X, T)$ which is a branched covering system of the $2-$sphere $\SS^2$. 

Arnold's cat map is an example of a homeomorphism belonging to the class of hyperbolic toral automorphisms. An analogue of hyperbolic toral automorphisms on the surface of genus one $\TT^2$ is the (linear) pseudo-Anosov diffeomorphisms on hyperelliptic translation surfaces of genus $g \geq 2$. These two classes of dynamical systems are branched coverings of $\SS^2$, and in fact, they can replace Arnold's cat map in the construction. Another family of translation surfaces \textit{$n-$Chamanara surface} and its homeomorphism \textit{$n-$baker map} turns out to be a potential candidate worth investigation.

An {\bf $\alpha-$Chamanara surface $C_\alpha$}, $\alpha \in (1, \infty)$, is an infinite genus translation surface introduced by Reza Chamanara in \cite{Cha}. Roughly, one can obtain the surface by assigning side identifications to $[0,1]^2$ as shown in Figure 1 below.
%-------------------------------------------
\begin{figure}[htp]
    \centering
    \includegraphics[width=6cm]{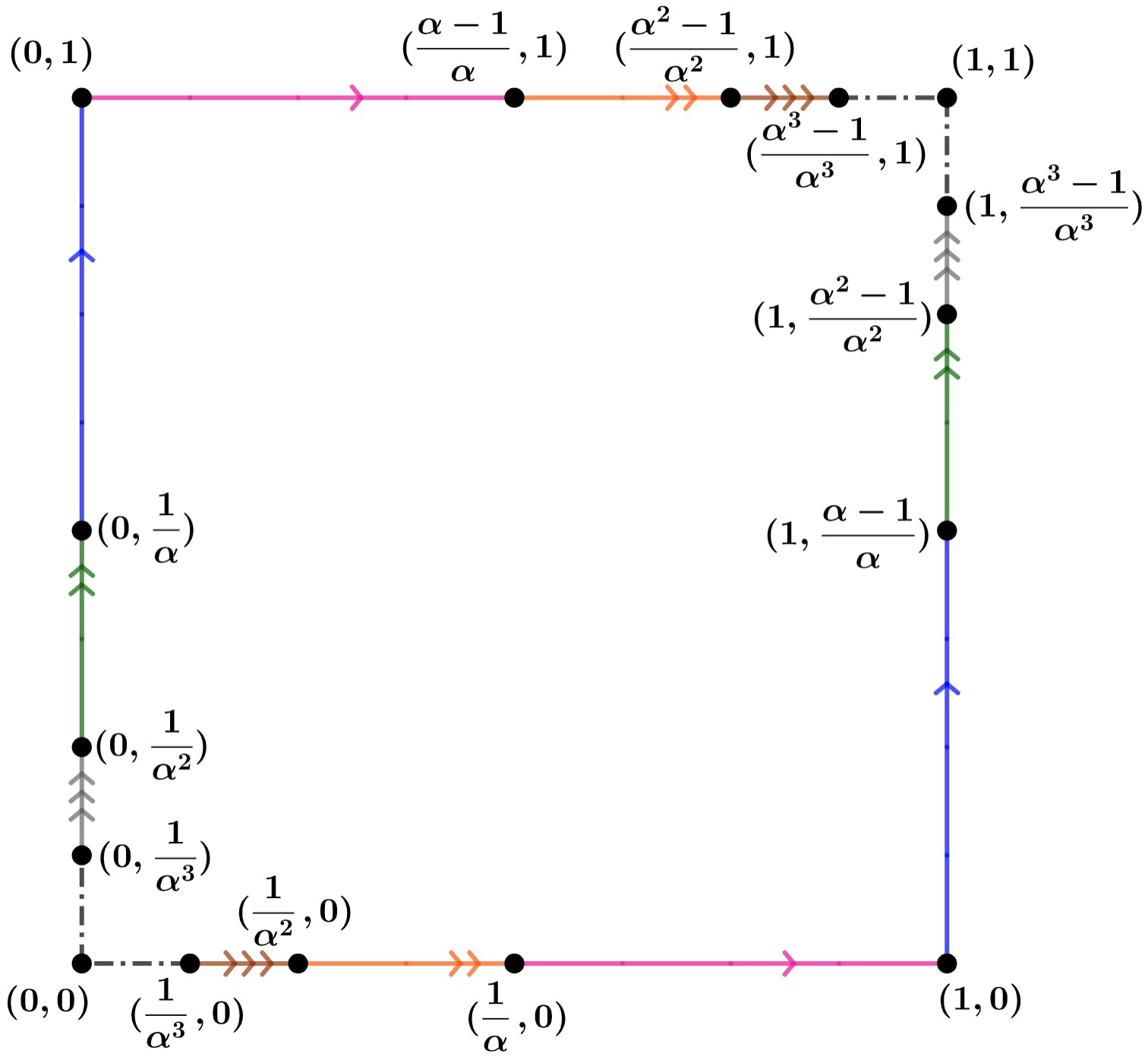}
    \caption{Side identifications impose on the square $[0,1]^2$. The quotient metric space with the single singular point $(0,0)$ is the $\alpha-$Chamanara surface, $\alpha \in (1, \infty)$.}
\end{figure}
%------------------------------------------
In the special case where $\alpha = n \in \NN$ with $n \geq 2$, the surface $n-$Chamanara $C_n$ admits a homeomorphism, the $n-$baker map $B_n$. This is a generalization of the well-known \textit{(2-)baker transformation $B_2$} defined on $[0,1)^2$ as
$$B_2(x,y) = \begin{cases}
(2x, y/2), & \mbox{if} \ 0 \leq x < 1/2, 0 \leq y < 1\\
(2x-1, \frac{1}{2}(y+1)), & \mbox{if} \ 1/2 \leq x < 1, 0 \leq y < 1.
\end{cases}$$
A rotation by $\pi$ radians $R$ centered at $(0.5,0.5)$ induces an equivalence relation $\sim_R$ on $C_n$. For the case $n=2$, Chamanara, Gardiner and Lakic verified in \cite{CGL} that the quotient space $C_2/\sim_R$ is a $2-$sphere $\SS^2$ (they point to an argument of de Carvalho and Hall in \cite{AT} where the result is an application of Moore's theorem) and there exists a homeomorphism on $C_2/\sim_R$ which commutes with $B_2$ via the projection $P_2 : C_2 \rightarrow C_2/\sim_R.$ We reprove this fact for $(C_2, B_2)$ and also extend the result to the pair $(C_n, B_n)$ for $n \geq 3$. Note that our approach is distinct from the one used in \cite{CGL}.
%----------------thm A--------------------
\begin{thmx}\label{thmx:00A}
For each $n \geq 2$, the dynamical system $(C_n, B_n)$ admits an equivalence relation $\sim_R$ given by the rotation by $\pi$ radians. The quotient space $Q_n := C_n/\sim_R$ is homeomorphic to $\SS^2$. The $n-$baker map descends down to $Q_n$ and induces a homeomorphism $T_n$ on $Q_n$ such that $P_n \circ T_n = B_n \circ P_n$ where $P_n : C_n \rightarrow Q_n$ is the natural projection.
\end{thmx}
%-----------------------------------------
As one sees in the discussion of Boro\'nski and Oprocha's construction given in Section $2$, there are additional required conditions on the initial system $(X,T)$ besides factoring to give a system on $\SS^2$:
\begin{itemize}
    \item The set of periodic points of the map $T$, $\text{Per}(T),$ is dense in $X$.
    \item When the equivalence relation inducing the quotient sphere is a local isometry, the homeomorphism $T$ is an affine diffeomorphism.
\end{itemize}
It is a well-known result that any hyperbolic toral automorphism $F_A$ on $\TT^2$ and any pseudo-Anosov homeomorphism $T_\lambda$ on a hyperelliptic translation surface $X_g$ satisfy these conditions. For each $n \geq 2$, the pair $(C_n, B_n)$ is a factor of the system of \textit{bi-infinite sequence on $n$ symbols} and its \textit{left-shift map} $(\Sigma_n, \sigma_n)$. This gives that $\overline{\text{Per}(B_n)} = C_n$ which allows one to get the following conclusion.
%----------------------Thm B---------------
\begin{thmx}\label{thmx:00B}
Let $\Gamma := \{(\TT^2, F_A) \mid F_A: \TT^2 \rightarrow \TT^2 \ \text{is a hyperbolic toral automorphism}\},$  $\Omega := \{(X_g, T_\lambda) \mid T_\lambda : X_g \rightarrow X_g \ \text{is a pseudo-Anosov homeomorphism}\}, \Lambda := \{(C_n, B_n) \mid n \geq 2\}.$ 
Any $(Y, S) \in \Gamma \cup \Omega \cup \Lambda$ is a valid initial system to build an inverse limit system of the Sierpi\'nski carpet and its homeomorphism.
\end{thmx}
%-------------------------------------------
Furthermore, each of this topological dynamical system $(X, T)$ admits invariant measures for which we can calculate the metric entropy. 
%--------------------------THM C-------------
\begin{thmx}\label{thmx:00C}
The entropy values of the dynamical systems on the carpet built in Theorem \ref{thmx:00B} are
$$\log(|\lambda|), \log(\lambda') \ \text{or} \ -\sum_{i=0}^{n-1} p_i\log(p_i)$$
where $\lambda$ is the leading eigenvalue of a hyperbolic toral automorphism, $\lambda'$ is the expansion factor of a pseudo-Anosov diffeomorphism and $P = (p_0, p_1, \ldots, p_{n-1})$ is a probability vector (i.e. $0 \leq p_i <1$ for all $i$, and $\sum_{i=0}^{n-1} p_i = 1$). In particular, all positive real numbers are realized as entropy values of dynamical systems on the carpet.
\end{thmx}
%---------------------------------------------
Another advantage of Boro\'nski and Oprocha's construction is that each dynamical system on the carpet inherits various (topological) dynamical properties from the initial system. However, Boro\'nski and Oprocha showed that the system induced from Arnold's cat map loses the property called the {\bf Bowen specification property}. Their proof relies on two key facts:
\begin{itemize}
    \item A local behavior of orbits of points near a hyperbolic fixed point.
    \item A nontrivial result of Pfister-Sullivan on the density of ergodic measures in the space of invariant probability measures (see \cite{PS}). 
\end{itemize}
A simplified version of the proof of the failure of Bowen specification in \cite{BO}, without using Pfister-Sullivan's result \cite{PS}, is given as our final result. In fact, a weaker specification-like property called the {\bf approximate product property} is shown to fail on the systems on the carpet.
%---------------------------Thm D--------------
\begin{thmx}\label{thmx:00D}
Any dynamical system on the carpet constructed in this paper, regardless of its initial system, does not have the approximate product property.
\end{thmx}
%-----------------------------------------------
\textit{Acknowledgements.} This paper is based on the Ph.D. dissertation of the author. The author expresses the deepest gratitude to Prof. Thomas Andrew Schmidt (Oregon State University) for his guidance and valuable suggestions given throughout the author's Ph.D. study. A part of results related to the pair $(C_2, B_2)$ is inspired by the talk given by Dr. Meyer on $27^{\text{th}}$ April 2021 in the Quasiworld seminar entitled: \textit{The Solenoid, the Chamanara Space,  and Symbolic Dynamics}. For this particular matter, the author would like to specially acknowledge Dr. Daniel Meyer (University of Liverpool).
%%%%%%%%%%%%%%%%%%%%%%%%%%%%%%%%%%%%%%%%%%%%%%%%%%%%%%%%%%
\section{Preliminaries}
The following notions will be used in this paper:\\
\begin{gross}
\item [$\TT^2$]  The surface of genus $1$ (the $2-$torus).
\item [$F_A$] The hyperbolic toral automorphism given by $F_A(z) = Az \pmod{1}$.
\item [$X_g$] A hyperelliptic translation surface of genus $g \geq 2.$
\item [$T_\lambda$] A pseudo-Anosov diffeomorphism with the expansion factor $\lambda > 1.$
\item [$C_n$]   The $n-$Chamanara surface, $n \geq 2$.
\item [$B_n$] The $n-$baker transformation, $n \geq 2$.
\item [$\Sigma_n$] The space of bi-infinite sequences on $n$ symbols $\{0, 1, \ldots, n-1\}.$
\item [$\sigma_n$] The left-shift homeomorphism on $\Sigma_n$.
\item [$\text{Per}(f)$] The set of periodic points of the function $f$.
\item [$\Bcal(X)$] The Borel $\sigma-$algebra on the topological space $X$.
\item [$h_\mu(f)$] The metric entropy of the function $f$ with respect to the measure $\mu$.
\item [$(S_\infty(T), H_\infty(T))$] The dynamical system on the Sierpi\'nski carpet constructed from the initial system $(X, T)$.
\end{gross}
\vspace{0.2cm}
%%%%%%%%%%%%%%%%%%%%%%%%%%%%%%%%%%%%%%%%%%%%%%%%%%%%%%
\subsection{Dynamical systems} We briefly introduce two notions of dynamical systems used in this paper. Refer to references as \cite{KH} and \cite{Ro} for a more thorough discussion.
\begin{itemize}
    \item {\bf (Topological dynamical systems)} A pair $(X,T)$ consisting of a topological space $X$ and a continuous map $T : X \rightarrow X$ is called a {\bf topological dynamical system}.
    \item {\bf (Measure-theoretic dynamical systems)} A quadruple $(X, T, \Fcal, \mu)$ consisting of a topological space $X$, a continuous function $T : X \rightarrow X$, a $\sigma-$algebra $\Fcal$ on $X$ and a measure $\mu : \Fcal \rightarrow [0, \infty)$ satisfying that $\mu(A) = \mu(T^{-1}(A)) \ \text{for all} \ A \in \Fcal \ \text{(this condition is called}\ \mu \ \textit{is} \ T-\text{invariant})$ is called a {\bf measure-theoretic dynamical system.} 
\end{itemize}
In the latter case, without specifying otherwise, we assume that $\Fcal = \Bcal(X)$. Also, we suppresses the words \textit{topological} or \textit{measure-theoretic} whenever possible. 

There is one particular technique of achieving a measure-theoretic structure for a given pair $(X,T)$ called a {\bf push-forward of a measure-theoretic structure}.
\begin{lemma}[A push-forward of a measure-theoretic structure] \label{lem:201} Let $(X, T)$ and $(Y, S)$ be dynamical systems. Assume that $(Y,S)$ is a factor of $(X,T)$ via $\pi : X \rightarrow Y.$ If in addition $(X, T, \Bcal(X), \mu_X)$ is a dynamical system, then there is an extension system $(Y, S, \Bcal(Y), \mu_Y)$, which is a factor of $(X, T, \Bcal(X), \mu_X)$, given by $\mu_Y := \mu_X \circ \pi^{-1}$.
\end{lemma}
Lemma \ref{lem:201} plays a crucial role in extending the pair $(C_n, B_n)$ for each $n \geq 2$. One shall also see that it is the first step to develop the result \textit{all positive real numbers can be achieved as entropy values of dynamical systems on the carpet}.

The main dynamical systems in this paper are $(C_n, B_n)$ for all $n \geq 2$. We also discuss briefly dynamical systems of $(\TT^2, F_A), (X_g, T_\lambda)$ and $(\Sigma_n, \sigma_n)$ for $g, n \geq 2.$ One can study definitions and facts regarding the pair $(\TT^2, F_A)$ in \cite{Ro}. Standard references for $(X_g, T_\lambda)$ include \cite{Zo} and \cite{FM}, with an exception that we prefer an equivalent definition of pseudo-Anosov diffeomorphisms stated in \cite{FSE}. For basic notions and results related to symbolic dynamics, consult \cite{KH}.
%%%%%%%%%%%%%%%%%%%%%%%%%%%%%%%%%%%%%%%%%%%%%%%%%%%%%
\subsection{Inverse limit systems}
An important result used both in \cite{BO} and here is Brown's inverse limit theorem (see \cite{MB}). It concerns a recognition of inverse limit systems via a notion of near homeomorphisms. 
%------------------------------------------------
\begin{definition}\label{def:201}
Let $(X_i)_{i \in \NN}$ be a sequence of compact metric spaces. Let $f_i : X_i \rightarrow X_{i-1}$ be a continuous map for all $i \geq 2$. The subspace $\text{Lim}(X_i, f_i)$ of the product space $\prod_{i=1}^\infty X_i$ defined as $$\text{Lim}(X_i, f_i) := \Big\{(u_i) \in \prod_{i=1}^\infty X_i : f_{ij}(u_j) := f_{i+1} \circ f_{i+2} \circ \cdots \circ f_j(u_j) = u_i \ \text{for all} \ i \leq j \ \text{with} \ f_{ii} = \text{id}_{X_i}\Big\}$$ is called the {\bf limit space} of the inverse system $(X_i, f_i).$
\end{definition}
%-------------------------------------------------
%-------------------------------------------------
\begin{definition}\label{def:202}
Let $(X,d)$ be a metric space. Then a map $f : X \rightarrow X$ is call a {\bf near homeomorphism} if for any $\epsilon > 0$, there exists a homeomorphism $h_\epsilon : X \rightarrow X$ such that $$ ||h_\epsilon - f|| := \sup_{x \in X}d(h_\epsilon(x), f(x)) < \epsilon.$$
\end{definition}
%------------------------------------------------
\begin{proposition} \label{prop:201} (Theorem $4$ in \cite{MB}) Let $(X_i)$ be a sequence of topological spaces. Assume that there exists a compact metric space $X$ such that $X_i$ is homeomorphic to $X$ for all $i \in \NN.$ If for all $i \geq 2, f_i : X_i \rightarrow X_{i-1}$ is a near homeomorphism, then $\text{Lim}(X_i,f_i)$ is homeomorphic to $X.$
\end{proposition}
%------------------------------------------------
%%%%%%%%%%%%%%%%%%%%%%%%%%%%%%%%%%%%%%%%%%%%%%%%%%%%%%%
\subsection{Boro\'nski and Oprocha's inverse limit construction} Core steps of Boro\'nski and Oprocha's inverse limit construction found in \cite{BO} are listed here. \\
{\bf Step 1:} We start with a \textit{system of Arnold's cat map on $\TT^2$}, $F_C(x,y)^t = \begin{pmatrix}
2 & 1\\
1 & 1
\end{pmatrix}(x,y)^t \pmod 1$. This is a well-known dynamical system with various intriguing properties: one of them is that \textit{it is a branched covering system of the $2-$sphere $\SS^2$}. In particular, one defines an equivalence relation $(x,y) \sim_{\TT^2} -(x,y)$ and forms a quotient space $\Scal_0 = \TT^2/\sim_{\TT^2}.$ It turns out that the quotient $\Scal_0$ is homeomorphic to $\SS^2$, and there exists a homeomorphism $H_0 : \Scal_0 \rightarrow \Scal_0$ such that $\pi \circ F_C = H_0 \circ \pi$ where $\pi : \TT^2 \rightarrow \Scal_0$ is the quotient map. Observe that there are four branch points $\Ccal = \{(0,0), (1/2,0), (0, 1/2), (1/2, 1/2)\}$ corresponding to $\sim_{\TT^2}.$\\
{\bf Step 2:} Since $\overline{\text{Per}(F_C)} = \TT^2$, $\overline{\text{Per}(H_0)} = \Scal_0$ via the semiconjugacy $\pi$. Moreover, using the facts that $\Scal_0$ is a connected second-countable Hausdorff topological space, one can decompose $\text{Per}(H_0)$ as $\text{Per}(H_0) = \Ocal \sqcup \Ocal'$ 
where $H_0(\Ocal) = \Ocal, H_0(\Ocal') = \Ocal', \Ocal \cap \pi(\Ccal) = \emptyset$ and $\overline{\Ocal} = \overline{\Ocal'} =\Scal_0.$\\
{\bf Step 3:} We further decompose $\Ocal = \sqcup_{i=1}^\infty O_i$ where each $O_i$ is a full orbit of the periodic point $z_i$ with the period length $n_i$. Then a technique called \textit{a blow up of a point on a (translation) surface} is applied. Roughly, given a point $c \in \Scal_0$, a {\bf blow up of a point $c$} is a pair $(\text{Blo}_c(\Scal_0), \pi_c)$ consisting of a topological space $\text{Blo}_c(\Scal_0)$ and a continuous function $\pi_c : \text{Blo}_c(\Scal_0) \rightarrow \Scal_0$ (called {\bf a collapsing map at $c$}) satisfying that 
\begin{itemize}
    \item $\pi_c^{-1}(\{c\})$ is a circle of directions centered at $c$ (i.e. a topological copy of the $1$-sphere $\SS^1$) denoted by $\SS^1(c)$,
    \item $\pi_c|_{(\text{Blo}_c(\Scal_0) \setminus\SS^1(c))}$ is a homeomorphism.
\end{itemize}
Basically, blowing up a point $c \in \Scal_0$ is a way to create a topological space from which an interior of a closed disc centered at $c$ is removed. The collapsing map $\pi_c$ gives suitable identifications of points between $\Scal_0$ and $\text{Blo}_c(\Scal_0)$.\\
{\bf Step 4:} We first blow up points in $O_1$ to create a space $\Scal_1$ and a continuous map $\pi_1 : \Scal_1 \rightarrow \Scal_0$. Via the collapsing map $\pi_1$ and the fact that \textit{$F_C$ preserves local radial lines}, a homeomorphism $H_1 : \Scal_1 \rightarrow \Scal_1$ satisfying that $\pi_1 \circ H_1 = H_0 \circ \pi_1$ is induced. \\
{\bf Step 5:} An induction based on Step 4 is performed. In fact, if a triple $(\Scal_k, H_k, \pi_k)$ consisting of a space $\Scal_k$, a homeomorphism $H_k : \Scal_k \rightarrow \Scal_k$ satisfying $\pi_k \circ H_k = H_{k-1} \circ \pi_k $ and a collapsing map $\pi_k : \Scal_k \rightarrow \Scal_{k-1}$ is already created by blowing up points belonging to $O_k$, then we blow up points in $O_{k+1}$ to create a triple $(\Scal_{k+1}, H_{k+1}, \pi_{k+1})$.\\
{\bf Step 6:} As a consequence, a limit set $$S_\infty(F_C) = \Big\{(u_i)_{i \geq 0} \in \prod_{i=0}^\infty \Scal_i \Big| \pi_j(u_j) = u_{j-1}\Big\}$$ is created. The result of Brown in Proposition \ref{prop:201} is applied to show that $S_\infty(F_C)$ is a Sierpi\'nski curve. A function $H_\infty(F_C) : S_\infty(F_C) \rightarrow S_\infty(F_C)$ defined by $$H_\infty(F_C)(u_i)_{i \geq 0} = (H_i(u_i))_{i \geq 0}$$ is a homeomorphism on $S_\infty(F_C)$.
\vspace{0.4cm}

As we seek to generalize this construction, we give a list of important properties needed in the construction:
\begin{itemize}
    \item One chooses \textit{an initial system $(X, T)$} with nice dynamical properties.
    \item In particular, the system $(X,T)$ admits an involution $\Ical$ on $X$ such that $(X, T)$ is a branched covering system of the $2-$sphere $\SS^2$ via $\Ical$. That is, the quotient $\Scal_0 = X/\Ical$ is homeomorphic to $\SS^2$, and there exists a homeomorphism $H_0$ on $\Scal_0$ such that $\pi \circ H_0 = T \circ \pi.$
    \item The set $\text{Per}(T)$ is dense in $X$ (and hence so is $\text{Per}(H_0)$ in $\Scal_0$).
    \item The homeomorphism $T$ preserves local radial lines. Note that any affine diffeomorphism of a plane possesses this property.
\end{itemize}

Since Arnold's cat map $F_C$ belongs to the class of hyperbolic toral automorphisms, it is not hard to believe, and indeed it is easily shown, that any $F_A$ satisfies all the requirements. Though less trivial, any pair $(X_g, T_\lambda)$ have all required properties as well (we discuss this formally in Subsection 3.1). We check that each pair $(C_n, B_n), n \geq 2,$ has all these properties in Subsection 3.2.
%%%%%%%%%%%%%%%%%%%%%%%%%%%%%%%%%%%%%%%%%%%%%%%%%%%%%%%
\subsection{Metric entropy of dynamical systems} This paper deals with a metric entropy of a given system $(X, T, \Bcal(X), \mu_X)$. In particular, the two following propositions below are sufficient to obtain entropy values of dynamical systems on the carpet. The first one is Proposition 4.3.16 in \cite{KH}. 
%--------------------------------------------------
\begin{proposition}\label{prop:202} Given dynamical systems $(X, T, \Bcal(X), \mu_X)$ and $(Y, S, \Bcal(Y),$ $\mu_Y)$ such that the latter is a metric factor of the former system, then $h_{\mu_Y}(S) \leq h_{\mu_X}(T).$ In addition, if the two systems are isomorphic, then $h_{\mu_Y}(S) = h_{\mu_X}(T).$
\end{proposition}
%---------------------------------------------------
 Given that the entropy $h_{\mu_X}(T)$ is known, one obtains an upper bound $h_{\mu_Y}(S) \leq h_{\mu_X}(T)$. In the special case when $X$ is a compact metric space, the work of F. Ledrappier and P. Walters in \cite{LW} gives a lower bound for $h_{\mu_Y}(S)$. In many cases, the lower bound is obtained and is equal to $h_{\mu_X}(T)$ itself implying that $h_{\mu_Y}(S) = h_{\mu_X}(T).$
%-----------------------------------------------------
\begin{definition}\label{def:203}
Let $(X, d)$ be a metric space and $T : X \rightarrow X$ be uniformly continuous. 
\begin{itemize}
    \item A set $F$ is said to {\bf $(n,\epsilon)-$span} a set $K$ if for each $x \in K$, there is $y \in F$ satisfying $d(T^j(x), T^j(y)) \leq \epsilon$ for all $0 \leq j < n.$
    \item For a compact set $K \subseteq X$, let $r_n(\epsilon, K)$ be the smallest cardinality of any set $F$ which $(n, \epsilon)-$spans $K$. Then define $\overline{r}_T(\epsilon, K) := \limsup_{n \rightarrow \infty} \frac{\log r_n(\epsilon, K)}{n}.$
    \item For $K \subseteq X$ compact, define $h(T,K) = \lim_{\epsilon \rightarrow 0} \overline{r}_T(\epsilon, K).$
\end{itemize}
\end{definition}
%------------------------------------------------------
%------------------------------------------------------
\begin{proposition}\label{prop:203} (Ledrappier-Walters \cite{LW})
Let $X$ and $Y$ be compact metric spaces. Let $T : X \rightarrow X$ and $S : Y \rightarrow Y$ be continuous. Assume that there is a continuous map $\pi : X \rightarrow Y$ such that $\pi \circ T = S \circ \pi$. Then $$\sup_{\mu : \mu \circ \pi^{-1} = \nu}h_\mu(T) = h_\nu(S) + \int_Y h(T, \pi^{-1}(y)) \ d\nu(y).$$ Additionally, if $|\pi^{-1}(y)| < \infty$ for all $y \in Y$, then $$\sup_{\mu : \mu \circ \pi^{-1} = \nu} h_\mu(T) = h_\nu(S).$$
\end{proposition}
%-------------------------------------------------------
%%%%%%%%%%%%%%%%%%%%%%%%%%%%%%%%%%%%%%%%%%%%%%%%%%%%%%
\subsection{The Bowen specification property} Bowen specification property is a \textit{strong shadowing property} introduced by Rufus Bowen in 1971 (see \cite{Bow}). Its usefulness has been evident among dynamicists, especially in the field of hyperbolic dynamics. See \cite{KLO} for a great overview on this topic. We define formally here only two specification-like properties: the Bowen specification property and the approximate product property.
%--------------------------------------------
\begin{definition}\label{def:204}
Let $(X, d)$ be a compact metric space and let $T : X \rightarrow X$ be a continuous surjective map. 
\begin{itemize}
    \item The pair $(X, T)$ has the {\bf (Bowen) specification property} if for any $\epsilon > 0$, there exists a positive integer $N(\epsilon)$ such that for any integer $s \geq 2$, given any $s$ points $y_1, y_2, ..., y_s \in X$ and any sequence of $2s$ integers $0 = j_1 \leq k_1 < j_2 \leq k_2 < ... < j_s \leq k_s$ satisfying $j_{m+1} - k_m \geq N$ for all $m = 1, 2, ..., s-1$, there exists a point $x \in X$ (called a shadowing point) with the property for each positive integer $m \leq s$, $$d(T^i(x), T^i(y_m)) < \epsilon$$ for all $i = j_m, j_m+1, \ldots, k_m.$ 
    \item The pair $(X,T)$ has the {\bf approximate product property} if for all $\epsilon > 0, \delta_1 >0, \delta_2 > 0$, there exists $N = N(\epsilon, \delta_1, \delta_2) \in \NN$ such that for any $n \geq N$ and for any sequence of points $(x_i)_{i \geq 1} \subseteq X$, there are (a sequence of gaps) $(h_i)_{i \geq 1} \subseteq \NN_0$ satisfying that $n \leq h_{i+1} - h_i\leq n+n\delta_2$ and a point $x \in X$, for each $i \in \NN$ $$|\{0 \leq j \leq n-1 : d(T^{h_i+j}(x), T^j(x_i)) \geq \epsilon\}| < \delta_1n. $$ 
\end{itemize}
\end{definition}
Observe that compactness together with the Bowen specification imply the approximate product property (see \cite{PS}). The converse is false as discussed in \cite{KLO}. Another important fact is that the Bowen specification is invariant under a topological semiconjugacy. Since $(\Sigma_n, \sigma_n)$ and any subshift of finite type have the Bowen specification (see say \cite{Si}), hyperbolic toral automorphisms $F_A$, pseudo-Anosov diffeomorphisms $T_\lambda$ and the $n$-baker map $B_n$ all have the Bowen specification (the first two are factors of some shifts of finite type, and the latter, as we soon prove, is a factor of $(\Sigma_n, \sigma_n)$). So it is a reasonable question to ask if a dynamical system on the carpet built from an initial system with the specification still has the specification. The answer is negative for the systems induced from $(\TT^2, F_A), (X_g, T_\lambda)$ or $(C_n, B_n)$ (see Section 5). 
%-----------------------------------------
%%%%%%%%%%%%%%%%%%%%%%%%%%%%%%%%%%%%%%%%%%%%%%%%%%%%%%
\section{The pair $(C_n, B_n)$ as a branched covering system of $\SS^2$}
%%%%%%%%%%%%%%%%%%%%%%%%%%%%%%%%%%%%%%%%%%%%%%%%%%%%%%
\subsection{A brief remark on $(\TT^2, F_A)$ and $(X_g, T_\lambda), g \geq 2$} The result for a pair $(\TT^2, F_A)$ is immediate as all arguments and proofs in \cite{BO} using Arnold's cat map $F_C$ remain valid when one replaces $F_C$ by any $F_A$. 

For the pair $(X_g, T_\lambda)$, we collect here all necessary facts: 
\begin{itemize}
    \item The set $\text{Per}(T_\lambda)$ is dense in $X_g$ (see Proposition 9.20 in \cite{FLP}).
    \item It is from the definition of $X_g$ that it admits an involution $\Ical : X_g \rightarrow X_g$ such that $\Scal_0 = X_g/\Ical$ is a $2-$sphere.  Moreover, Lemma 2.3 in \cite{BL} gives that there is a homeomorphism $H_0$ on $\Scal_0$ which properly commutes with $T_\lambda$.
    \item The fact that $T_\lambda$ being an affine diffeomorphism gives that it preserves radial lines locally.
\end{itemize}
As a result, one concludes that both families of $(\TT^2, F_A)$ and $(X_g, T_\lambda)$ are valid as initial systems for Boro\'nski and Oprocha's inverse limit construction. 
%%%%%%%%%%%%%%%%%%%%%%%%%%%%%%%%%%%%%%%%%%%%%%%%%%%%%%
\subsection{The proof of the first main result}
We now verify the first main result corresponding to Theorem \ref{thmx:00A} and \ref{thmx:00B} above. We present the proof in three parts. The first part is to show that each pair $(C_n, B_n)$ is a factor of $(\Sigma_n, \sigma_n)$. This establishes that $\overline{\text{Per}(B_n)} = C_n$ and a way to obtain a measure-theoretic structure for the pair $(C_n, B_n)$ (which we discuss later in Section 4). After this we form a quotient space $Q_n$ of $C_n$ by an involution of a rotation by $\pi$ radians about $(0.5, 0.5)$. The second part is devoted to showing that $Q_n$ is a $2-$sphere. The last part verifies an existence of a homeomorphism $T_n$ on $Q_n$ commuting with $B_n$.
%%%%%%%%%%%%%%%%%%%%%%%%%%%%%%%%%%%%%%%%%%%%%%%%%%%%%
\subsubsection{The pair $(C_n, B_n)$ as a factor of $(\Sigma_n, \sigma_n)$}
Let $n \geq 2$ be fixed. We begin by listing each element $\omega = \overline{(x,y)} = \overline{(0.x_1x_2\ldots_n, 0.y_1y_2\ldots_n)} \in C_n$ explicitly 
$$\omega = \begin{cases}
\{\omega\} &\mbox{if} \ \omega \in (0,1)^2,\\
\{(0, 0.y_1y_2\ldots_n), (1, 0.y_1^-y_2\ldots_n)\} &\mbox{if} \ (x,y) \in I_1 \cup I_1', \\
\{(1, 0.n^-n^-\ldots n^-y_ky_{k+1}\ldots_n), (0, 0.00\ldots0y_k^+y_{k+1}\ldots_n)\} &\mbox{if} \ (x,y) \in I_k \cup I_k', k \geq 2 \\
\{(0.x_1x_2\ldots_n, 0), (0.x_1^-x_2\ldots_n, 1)\} &\mbox{if} \ (x,y) \in J_1 \cup J_1',\\
\{(0.n^-n^-n^-\ldots n^-x_kx_{k+1}\ldots_n, 1), (0.000\ldots0x_k^+x_{k+1}\ldots_n, 0)\} &\mbox{if} \ (x,y) \in J_k \cup J_k', k \geq 2
\end{cases}$$ where $ a^- := a - 1, a^+ := a+1, 0.a_1a_2\ldots_n := \sum_{i=1}^\infty \frac{a_i}{n^i}$, 
$$\overline{(0,0)} = \{(0,0), (1,1), (1,0), (0,1)\} \cup \Big\{\Big(0, \frac{1}{n^k}\Big), \Big(\frac{1}{n^k}, 0\Big), \Big(1, 1 - \frac{1}{n^k}\Big), \Big(1 - \frac{1}{n^k}, 1\Big) : k \in \NN\Big\} \ \text{and for} \ i \geq 1,$$ $$I_i =\{0\} \times \Big(\frac{1}{n^{i}}, \frac{1}{n^{i-1}}\Big), I_i' = \{1\} \times \Big(1-\frac{1}{n^{i-1}}, 1-\frac{1}{n^i}\Big), J_i = \Big(\frac{1}{n^{i}}, \frac{1}{n^{i-1}}\Big) \times \{0\}, J_i' = \Big(1-\frac{1}{n^{i-1}}, 1-\frac{1}{n^i}\Big) \times \{1\}.$$
%-------------------------------------------
\begin{figure}[htp]
    \centering
    \includegraphics[width=5cm]{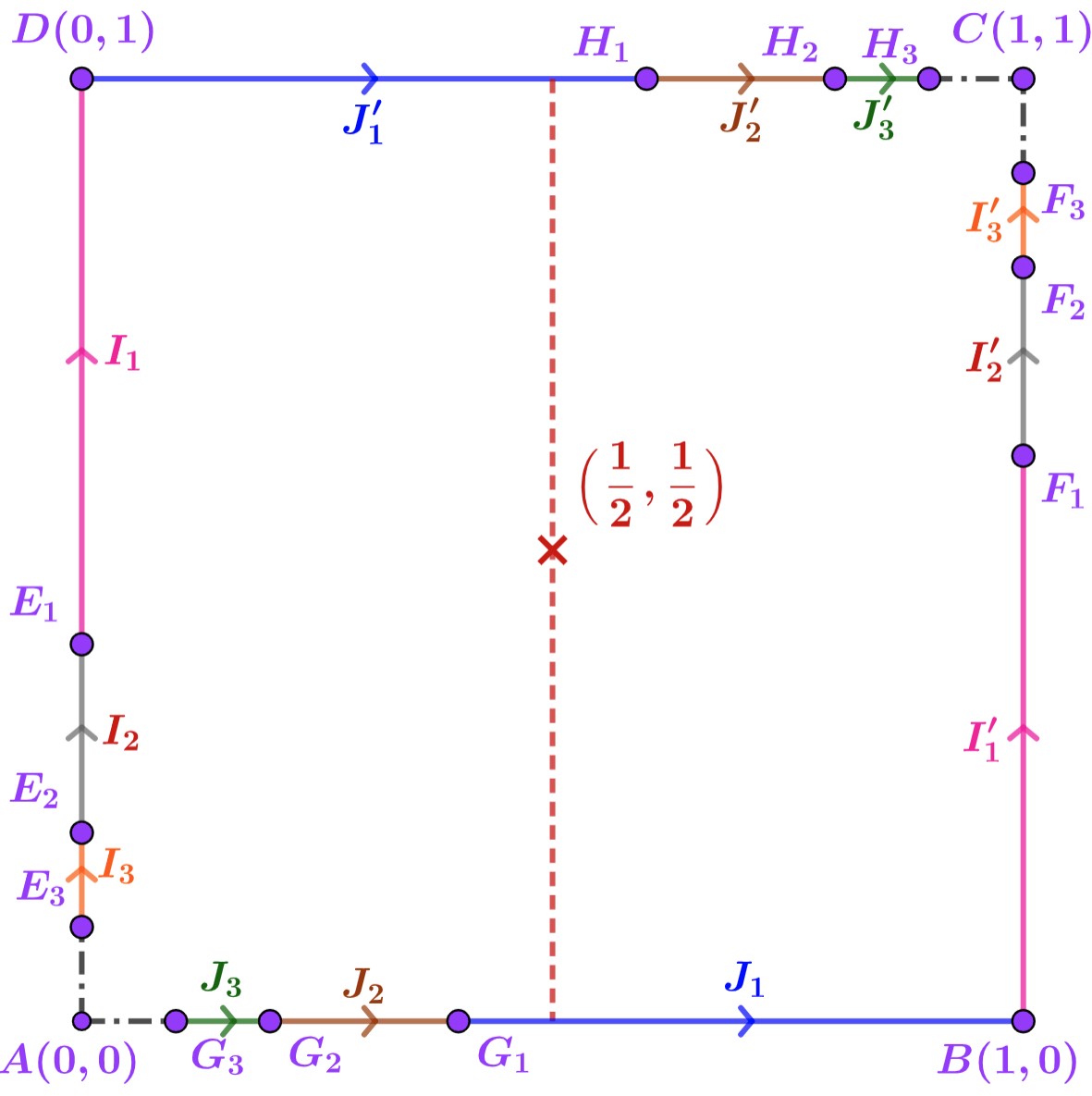}
    \caption{The surface $C_n$ described using identifications on $I_i, I_i', J_i$ and $J_j'$.}
\end{figure}
%------------------------------------------
One observes that for $(0,y) \in I_1$, $y = 0.y_1y_2\ldots_n$ where $y_1 \in \{1, 2, \ldots, n-1\}.$ The identification to $I_1'$ translates the $y$-coordinate down by $\frac{1}{n} = 0.1_n$. So $(0,y)$ is identified to $(1,y - 0.1_n) = (1, y_1^-y_2\ldots_n) \in I_1'.$ As for $k \geq 2,$ let $z = (0,y) \in I_k$. Then $y = 0.y_1y_2\ldots y_{k-1}y_k\ldots_n$ where $y_1 = y_2 = \ldots = y_{k-1} = 0$ and $y_k \in \{1, 2, \ldots, n-1\}.$ One sees that $n^{k-1}z$ shifts $z \in I_k$ to $n^{k-1}z \in I_1$.
Then $n^{k-1}z = (0,0.y_k\ldots_n)$ is identified with $w = (1, y_{k-1}^-y_k\ldots_n) \in I_1'.$ After that the $y-$coordinate of $w$ is retracted by $n^{1-k}$. It is lastly translated up to $I_k'.$ This gives that $z = (0.00\ldots0y_ky_{k+1}\ldots_n) \ \text{is identified with} \ (1, 0.n^-n^-\ldots n^-y_k^-y_{k+1}\ldots_n) \in I_k'.$
\vspace{0.4cm}

An {\bf n-baker map} is a function $B_n : [0,1)^2 \rightarrow [0,1)^2$ defined by $ B_n(x,y) = \Big(nx-(k-1), \frac{y+(k-1)}{n}\Big)$
for $\frac{k-1}{n} \leq x < \frac{k}{n}$ for each $k = 1, 2, \ldots, n.$ The definition of $B_n$ naturally extends to $C_n$ since $[0,1)^2$ is a fundamental domain of $C_n.$ We now verify that $(C_n, B_n)$ is a factor of $(\Sigma_n, \sigma_n)$. Then, as an immediate consequence, $\text{Per}(B_n)$ is dense in $C_n$.
%-----------------------lem 3-1 --------------------
\begin{lemma}\label{lem:301}
The function $B_n$ is a homeomorphism on $C_n$. The topological dynamical system $(C_n, B_n)$ is a factor of $(\Sigma_n, \sigma_n)$ via the map $P_n : \Sigma_n \rightarrow C_n$ defined by $$P_n(b_m;a_n) = \overline{(0.a_1a_2\ldots_n, 0.b_1b_2\ldots_n)}.$$ Consequently, $\text{Per}(B_n)$ is dense in $C_n.$
\end{lemma}
\begin{proof}
Define vertical strips $\Ecal_i$ and vertical right lines $\Vcal_j$ as, for each $1 \leq i \leq n, 1 \leq j \leq n-1$,
$$\Ecal_i = \Big(\frac{i-1}{n}, \frac{i}{n}\Big) \times (0,1) \ \text{and} \ \Vcal_j = \Big\{\frac{j}{n}\Big\} \times (0,1).$$
Then for each $1 \leq i \leq n$, define $f_i$ on $\overline{\Ecal_i}$ by $f_i(x,y) = \big(nx-(i-1), \frac{y+(i-1)}{n}\big).$ Then $f_i$ is continuous on $\overline{\Ecal_i}$ for all $1 \leq i \leq n.$ Since $\overline{\Ecal_j} \cap \overline{\Ecal_{j+1}} = \Vcal_i$ for $1 \leq j \leq n-1$, and $f_j(\Vcal_j) = f_{j+1}(\Vcal_j)$ by the identifications on $I_1$ and $I_1'$, $B_n$ is well-defined and continuous on $C_n$. Analogously, one uses the identifications on $J_1$ and $J_1'$ together with an introduction of suitable horizontal strips and horizontal upper lines to show that $B_n^{-1}$ is well-defined and continuous on $C_n.$

To show the continuity of $P_n$, note that $P_n = q_n \circ p_n$ where $p_n : \Sigma_n \rightarrow [0,1]^2, q_n : [0,1]^2 \rightarrow C_n$.
It then suffices to show that $p_n(a_m;b_n) = (0.a_1a_2\ldots_n, 0.b_1b_2\ldots_n)$ is continuous, which is immediate from the fact that if $d_{\Sigma_n}((a_m;b_n), (c_m;d_n)) \leq 2^{-k}$, then 
\begin{align*}
    \Big(\sum_{i=1}^\infty \frac{a_i-c_i}{n^i}\Big)^2 + \Big(\sum_{i=1}^\infty \frac{b_i-d_i}{n^i}\Big)^2 = \Big(\sum_{i > k} \frac{a_i-c_i}{n^i}\Big)^2 + \Big(\sum_{i > k} \frac{b_i-d_i}{n^i}\Big)^2 \leq 2 \Big(\sum_{i \geq k+1}\frac{n-1}{n^i}\Big)^2 = \frac{2}{n^{2k}}.
\end{align*}
Lastly, we verify the commutativity of the diagram. For $0 \leq i \leq n-1$, let
$$\VV_i = \{(y_m;x_n) : x_1 = i, x_j = n-1 \ \mbox{for all} \ j \geq 2 \} \cup \{(y_m;x_n) : x_1 = i+1, x_j = 0 \ \mbox{for all} \ j \geq 2\}.$$
Define 
$\EE_i = \{(y_m;x_n) \in \Sigma_n : x_1 = i\} \setminus (\VV_{i-1} \cup \VV_i)$ and $\EE_0 = \{(y_m;x_n) \in \Sigma_n : x_1 = 0\} \setminus (\VV_0 \cup \{(y_m;x_n) : x_n = 0 \ \mbox{for all} \ n \in \NN\})$ for each $1 \leq i \leq n-1$. One regards these $\EE_i$ and $\VV_j$ as notions of vertical strips and vertical right lines in the space $\Sigma_n.$ 
Note that $\Sigma_n = \sqcup_{i = 0}^{n-1} (\EE_i \sqcup \VV_i) \sqcup \{(y_m;x_n) : x_n = 0 \ \mbox{for all} \ n \in \NN\}.$\\
Let $z = (y_m;x_n) \in \Sigma_n.$ \\
{\bf Case 1:} Assume that $z \in \EE_k$ for some $0 \leq k \leq n-1$. \\
Then $B_nP_n(z) = B_n(0.kx_2\ldots_n, 0.y_1y_2\ldots_n) = (0.x_2x_3\ldots_n, 0.ky_1y_2\ldots_n) = P_n\sigma_n(z).$\\
{\bf Case 2:} Assume $z \in \VV_k$ for some $0 \leq k \leq n-2$.\\
Then 
$$
B_nP_n(z) = 
B_n(0.k^+_n, 0.y_1y_2\ldots_n) = (0, 0.k^+y_1y_2\ldots_n) = P_n(\ldots y_2y_1k^+;0\ldots0\ldots) = P_n\sigma_n(z), \ \text{and}
$$ 
\begin{align*}
    B_nP_n(w) &= B_n(0.kn^-\ldots n^-\ldots_n, 0.y_1y_2\ldots_n) = (0.n^-\ldots n^-\ldots_n, 0.ky_1y_2\ldots_n)\\
    &= P_n(\ldots y_2y_1k;n^-\ldots n^-\ldots) = P_n\sigma_n(w)
\end{align*}
where $z = (\ldots y_2y_1;k^+0\ldots 0 \ldots), w = (\ldots y_2y_2;kn^-\ldots n^-\ldots)$.\\
{\bf Case 3:} Assume that $z = (y_m;0)$. \\
Then $B_nP_n(z) = B_n(0, 0.y_1y_2\ldots_n) = (0, 0.0y_1y_2\ldots_n) = P_n\sigma_n(z).$\\
{\bf Case 4:} Assume that $z = (y_m;n^-\ldots n^-\ldots)$\\
Then $B_nP_n(z) = B_n(0.n^-\ldots n^-\ldots_n, 0.y_1y_2\ldots_n) = (0.n^-\ldots n^-\ldots_n, 0.n^-y_1y_2\ldots_n) = P_n\sigma_n(z).$\\
So the system $(C_n, B_n)$ is a factor of $(\Sigma_n, \sigma_n)$.
\end{proof}
%--------------------------------------------
%%%%%%%%%%%%%%%%%%%%%%%%%%%%%%%%%%%%%%%%%%%%%%%%%%%%%
\subsubsection{The quotient $Q_n$ is a $2-$sphere $\SS^2$}
As mentioned earlier, an involution $R$ on $[0,1]^2$ is a rotation by $\pi$ radians about the point $(0.5, 0.5)$. It is a straight-forward calculation to see that $C_n$ admits an equivalence relation $\omega \sim_R R(\omega).$ In fact, for $\omega_1 = \{(0, 0.y_1y_2\ldots_n), (1, 0.y_1^-y_2\ldots_n)\} \subseteq I_1 \cup I_1'$,
$$R(0,0.y_1y_2\ldots_n) = (1, 0.y_1^*y_2^*\ldots_n) \sim_{C_n} (0,0.(y_1^*)^-y_2^*\ldots_n) = (0,0.(y_1^-)^*y_2^*\ldots_n) = R(1,0.y_1^-y_2\ldots_n)$$ and for $\omega_2 = \{(1, 0.n^-n^-\ldots n^-y_ky_{k+1}\ldots_n), (0, 0.00\ldots0y_k^+y_{k+1}\ldots_n)\} \subseteq I_k \cup I_k' $ for some $k \geq 2,$
$$R(1, 0.n^-\ldots n^-y_k\ldots_n) = (0, 0\ldots 0 y_k^*\ldots_n) \sim_{C_n} (1, n^-\ldots n^-(y_k^+)^*\ldots_n) = R(0, 0.0\ldots 0y_k^+\ldots_n)$$ where $a^* := n^- - a = (n-1) - a.$ 

So we form a quotient space $Q_n := C_n/\sim_R$, which is now shown to be homeomorphic to $\SS^2$. The base case of $Q_2$ is first verified. We recall that this result was already known in \cite{CGL}, but the verification here is done via a new approach of building a certain inverse limit system which is proved to be a $2-$sphere using Brown's work in Proposition \ref{prop:201}.
\vspace{0.4cm}

The quotient $Q_2$ can be visualized by side identifications on $[0,0.5] \times [0,1]$ as shown below.
%-------------------------------------------
\begin{figure}[htp]
    \centering
    \includegraphics[width=6cm]{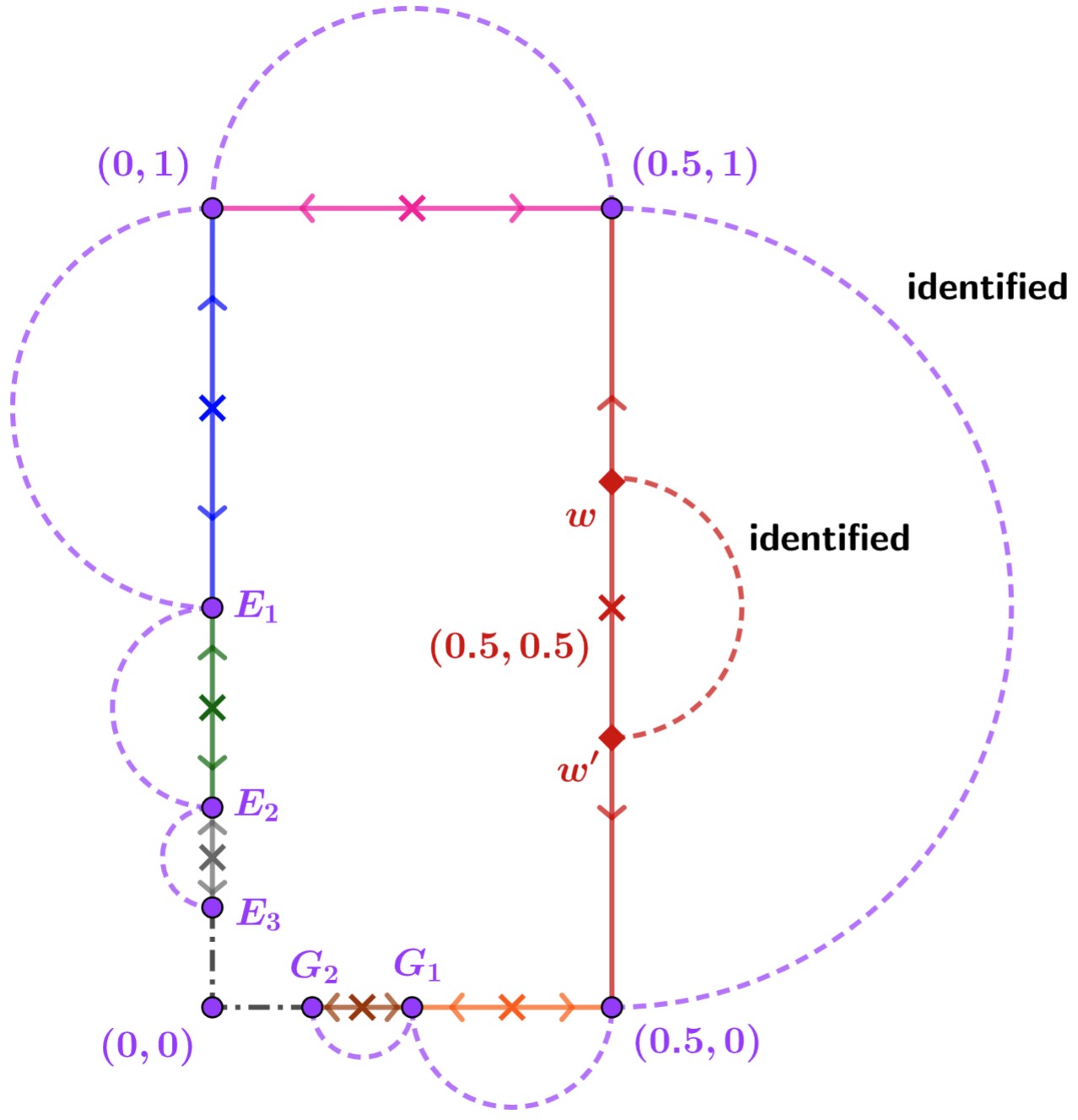}
    \caption{Side identifications on the boundary of $[0,0.5] \times [0,1]$ results in the quotient $Q_2$. One should compare the identifications and notations with the Figure 2.}
\end{figure}\\
%------------------------------------------
To show that $Q_2 \cong \SS^2$, we adopt the idea of \textit{zips} similar both to that used in Conway's zip proof of the classification of closed surfaces (see \cite{FW}) and to Veech's notion of zippered rectangles (see say \cite{Zo}). In particular, a zip represents an equivalence relation imposed on a certain part of a boundary. With the zip notation, one sees that the quotient $Q_2$ is homeomorphic to a $2-$sphere which is zipped up countably infinitely many times along the single (singular) vertex point, denoted by $\SS^2(Q_2)$ (see Figure 4 below). It suffices to show that $\SS^2(Q_2) \cong \SS^2.$
%--------------------------------------------
\begin{figure}[htp]
    \centering
    \includegraphics[width=8cm]{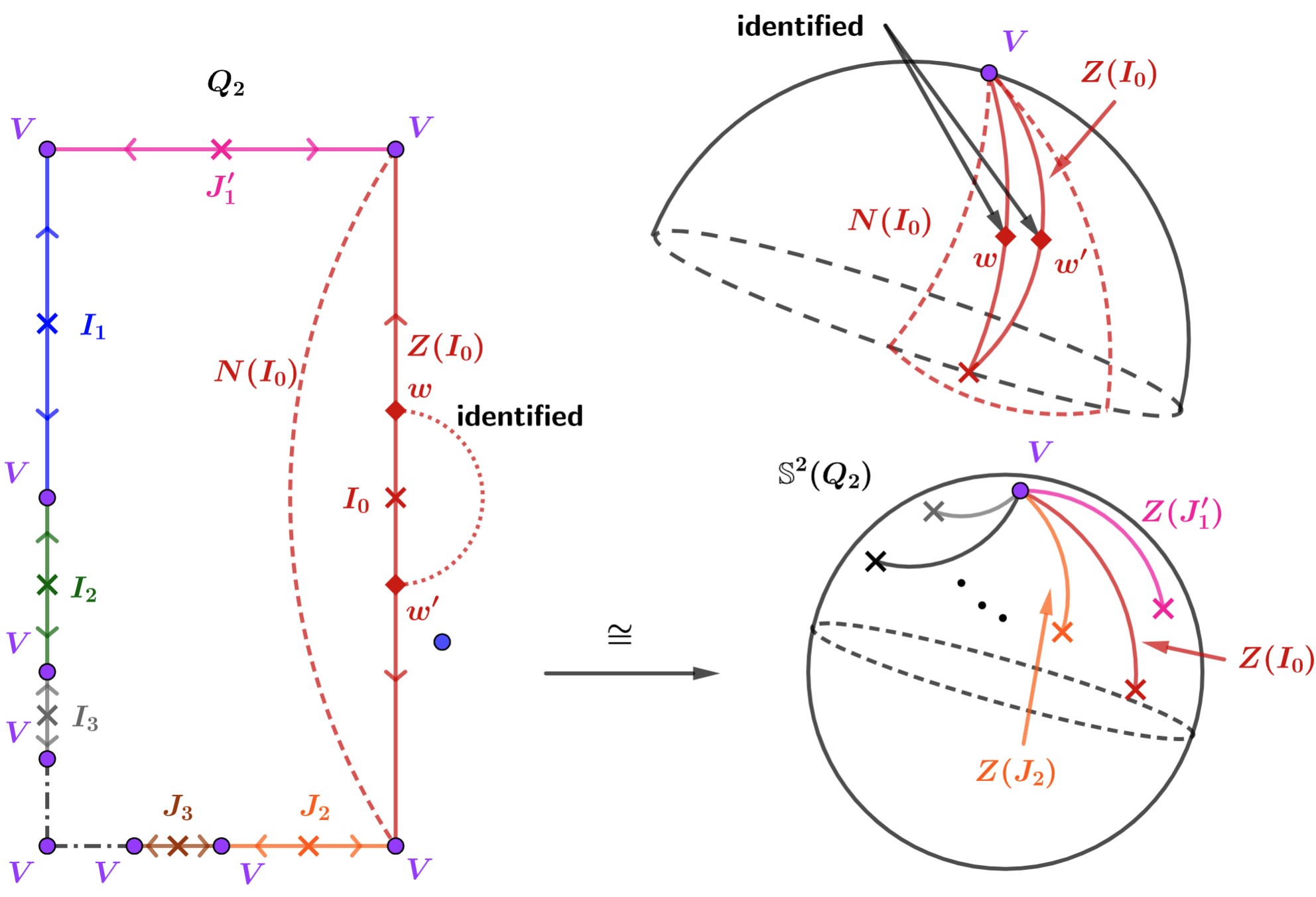}
    \caption{The side identification of $I_0$ is \textit{zipping up} along the edge of $I_0$. A neighborhood $N(I_0)$ of the zip $Z(I_0)$ is shown in the top right figure. Hence, the space $Q_2$ is homeomorphic to the 2-sphere with countably infinitely many zips of diminishing sizes, $\SS^2(Q_2)$.}
    \end{figure}
%------------------------------------------------

Observe that each zip $Z(P)$ has its own neighborhood called $N(P)$. These neighborhoods are pairwise disjoint $N(P) \cap N(Q) = \emptyset$ for zips $P \neq Q$. With this observation, we build an inverse limit system as follows.\\
{\bf Step 1:} Let $M_0$ be a $2-$sphere obtained by collapsing the boundary of $[0,0.5] \times [0,1]$ to a single point. Next, we create $M_1$ as a quotient space of $[0,0.5] \times [0,1]$ obtained by imposing the equivalence relation $\sim_R$ on $I_0$ and collapse $\partial([0,0.5]\times [0,1]) \setminus (I_0 \setminus \partial(I_0))$ to a single point. That is, $M_1$ is a $2-$sphere with a zip representing the relation imposed on $I_0$.\\
{\bf Step 2:} The space $M_2$ is a $2-$sphere with two zips created by imposing the relation $\sim_R$ on $I_0$ and $J_1'$, and collapsing $\partial([0,0.5] \times [0,1]) \setminus ((I_0 \setminus \partial(I_0)) \cup (J_1' \setminus \partial(J_1')))$ to a single point.\\
{\bf Step 3:} In general, we inductively create a new space $M_{k+1}$ from the already created space $M_k$ by imposing the relation $\sim_R$ to an additional part of $\partial([0,0.5] \times [0,1])$, and collapsing the rest of the boundary part to a single point. In particular, Table $1$ shows a pattern of imposing the relation $\sim_R$ to an additional part of the boundary for the first five steps.
\begin{table}[ht]
\begin{center}
\begin{tabular}{ |c|c|c|c|c|c|c| } 
\hline
{\bf Step $k^{th}$ of the construction} & {\bf $0^{th}-$step} & {\bf $1^{st}-$step} & {\bf $2^{nd}-$step} & {\bf $3^{rd}-$step} & {\bf $4^{th}-$step} & {\bf $5^{th}-$step}  \\
\hline
{\bf The $\sim_R$ is imposed on} & - & $I_0$ & $J_1'$ & $I_1$ & $J_2$ & $I_2$ \\
\hline
{\bf The created space} & $M_0$ & $M_1$ & $M_2$ & $M_3$ & $M_4$ & $M_5$ \\
\hline
\end{tabular}
\end{center}
\vspace{5pt}

\caption{A pattern of imposing $\sim_R$ to parts of $\partial([0,0.5] \times [0,1])$ for the first five steps.}
\end{table}\\
{\bf Step 4:} As a result, we have a sequence of topological spaces $(M_i)_{i \geq 0}$ such that each pair of consecutive spaces differ by a zip. The last requirement to build an inverse limit system of $(M_k)$ is to define a suitable continuous map $h_k : M_k \rightarrow M_{k-1}$ for $k \in \NN.$ The facts that each consecutive pair of spaces differ by a zip, called it $Z(P)$, and each zip $Z(P)$ has it own neighborhood $N(P)$ allow us to define the map $h_k$ locally in $N(P)$. That is, a map is the identity outside of $N(P)$. We give a precise definition below.

A function $h : \overline{N(P)} \rightarrow \overline{N(P)}$ which collapses the zip $Z(P)$ is defined as follows:\\ 
Let $h(Z(P)) = \{[(0,0)]\}$ and $h(z) = z$ for all $z \in \partial N(P).$ Then we define $h$ on each radial arc by scaling its length corresponding to how close it is to the zip $Z(P)$. The resulting function $h$ satisfies that
\begin{itemize}
    \item $h$ is a homeomorphism on $N(P)\setminus Z(P)$ and $h$ is continuous on $N(P)$,
    \item $h$ fixes the boundary $\partial N(P),$
    \item but, $h$ is not injective on $Z(P)$.
\end{itemize}
%------------------------------------------
 \begin{figure}[htp]
    \centering
    \includegraphics[width=7cm]{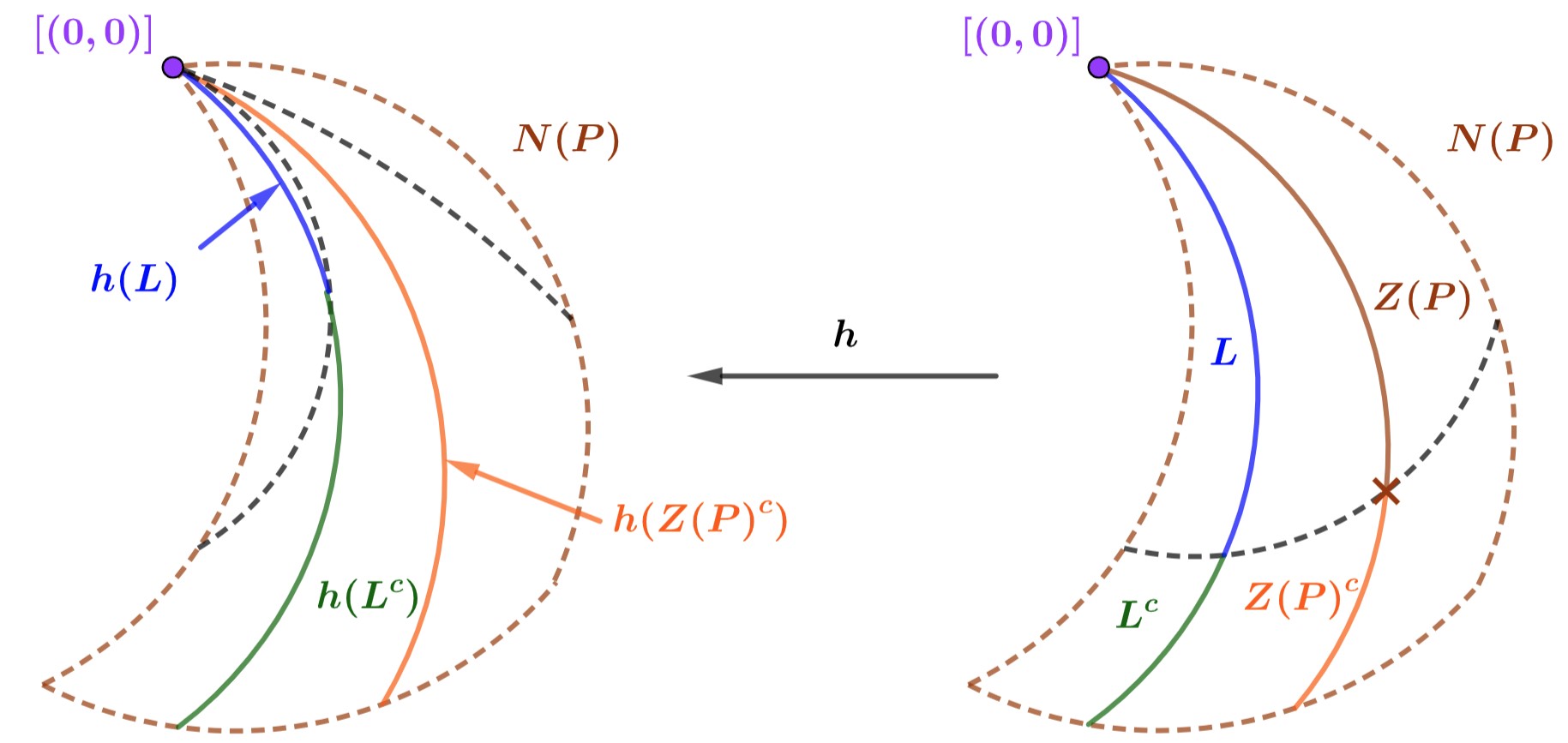}
    \caption{The continuous function $h$ collapses the zip $Z(P)$\,.}
    \end{figure}
%-------------------------------------------
%--------------------------------------
\begin{lemma}\label{lem:302} The continuous function $h : \overline{N(P)} \rightarrow \overline{N(P)}$ is a near homeomorphism.
\end{lemma}
\begin{proof}
For each small $\delta > 0,$ we can define a homeomorphism $h_\delta : \overline{N(P)} \rightarrow \overline{N(P)}$ as follows:
\begin{itemize}
    \item Let $\ell_\delta$ be the curve of length $\delta$ contains in $Z(P)$ and having $[(0,0)]$ as one of the endpoints. The map $h_\delta$ sends $Z(P)$ to the arc $\ell_\delta$ instead of the point $[(0,0)]$.
    \item The map $h_\delta$ is the identity on $\partial N(P)$.
    \item Let $|Z(P)|$ denote the length of the zip $Z(P)$. Each radial arc in $N(P) \setminus Z(P)$ is the union of an arc $L$ of length $|Z(P)|$ and its continuation $L^c.$ The map $h_\delta$ sends $L \cup L^c$ to itself by linearly contracting $L$ and expanding $L^c.$ We set the constant of contraction according to the angle between $L$ and $Z(P)$, and the fixed $\delta > 0.$
\end{itemize}
 A motivation for definition of $h_\delta$ is that as $\delta \rightarrow 0$, $h_\delta$ tries to be the same as $h$. So by the construction, for a fixed $\epsilon > 0$, then there exists $\delta > 0$ such that $||h_\delta - h|| < \epsilon.$ 
%--------------------------------------
%--------------------------------------------
 \begin{figure}[htp]
    \centering
    \includegraphics[width=7cm]{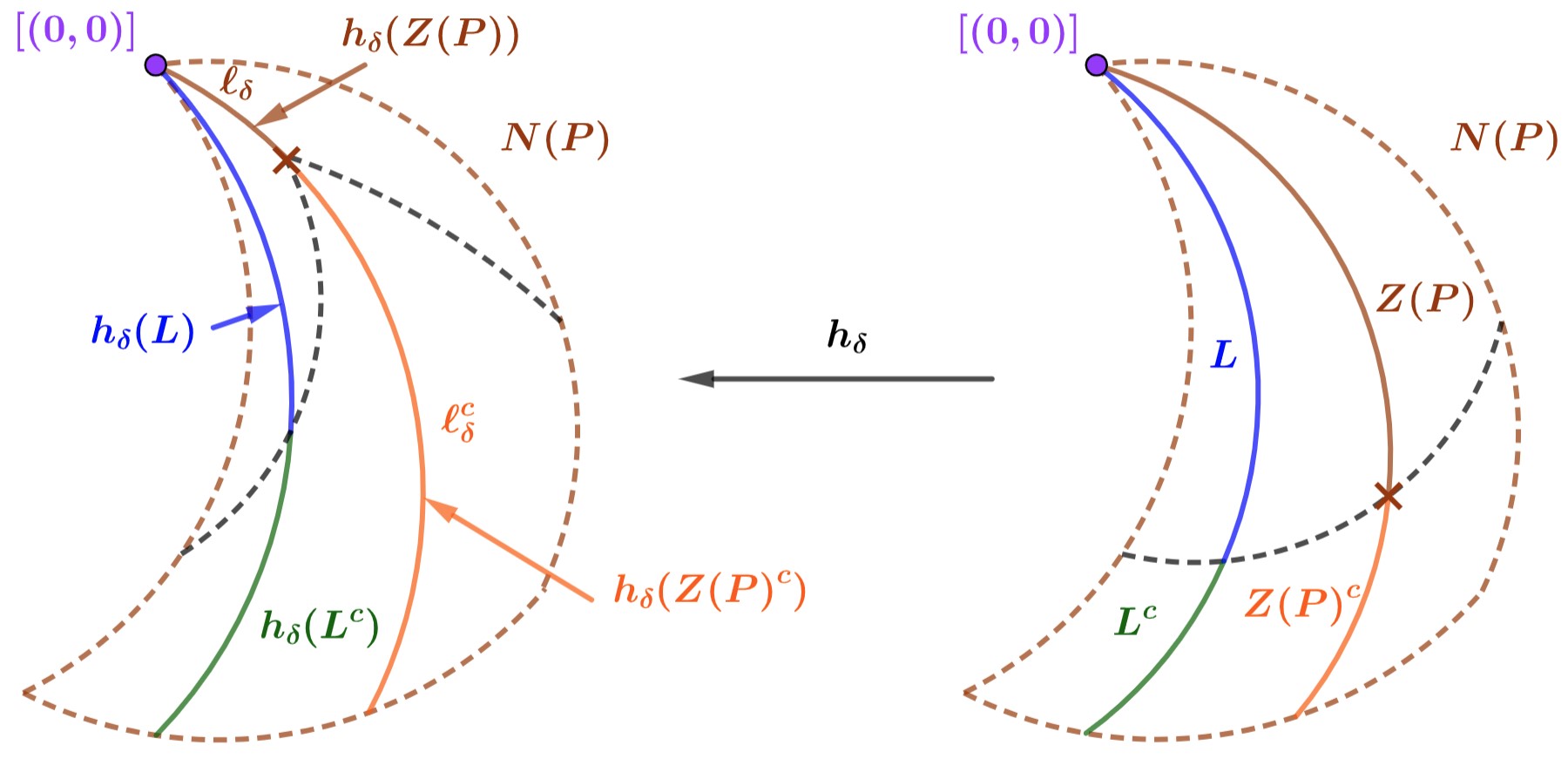}
    \caption{The homeomorphism $h_\delta$ decreases the size of the zip $Z(P)$\,.}
    \end{figure}
%-----------------------------------
\end{proof}
%----------------------------------------
Since $M_i \cong \SS^2$ for all $i \geq 0$, Proposition \ref{prop:201} and Lemma \ref{lem:302} yield that $$\text{Lim}(M) := \Big\{(z_i) \in \prod_{i=0}^\infty M_i : h_j(z_j) = z_{j-1}\Big\} \cong \SS^2.$$ 
%--------------------------------------
\begin{proposition}\label{prop:301}
The topological space $\SS^2(Q_2)$ is homeomorphic to $\SS^2.$ In particular, $Q_2$ is homeomorphic to $\SS^2.$
\end{proposition}
\begin{proof}
For each $i \geq 0$,  we then define a function $\phi_i : \SS^2(Q_2) \rightarrow M_i$ by
$$ \phi_i(z) =
\begin{cases}
h_k(z) &\text{for} \ z \in N(P_k), k > i \\
z &\text{otherwise}.
\end{cases}
$$
Then $\phi_i$ is continuous because each $h_k$ is identity for all $z \in \partial N(P_k), k > i$.\\
We now have a space $\SS^2(Q_2)$ and a family of continuous functions $(\phi_i : \SS^2(Q_2) \rightarrow M_i)_{i \geq 0}.$\\
If for $i \geq j,$ let $f_{ij} = h_i \circ h_{i-1} \circ \ldots \circ h_j : M_i \rightarrow M_j$ with $f_{ii} = \text{id}_{M_i}.$\\
Then $(\phi_i)$ is compatible with $(f_{ij})$ in the sense that for any $i \geq j$, $f_{ij}\phi_i = \phi_j.$ \\
By the universal property of an inverse limit of topological spaces, there exists a unique continuous function $F : \SS^2(Q_2) \rightarrow \text{Lim}(M)$ such that for each $i \geq 0,$ $\psi_i\circ F = \phi_i$  where $\psi_i : \text{Lim}(M) \rightarrow M_i$ is the $i-$coordinate projection.
Define a function $G : \text{Lim}(M) \rightarrow \SS^2(Q_2)$ as follows: Let $z \in \SS^2(Q_2)$.\\
{\bf Case 1:} $z \in \SS^2(Q_2) \setminus (\cup_{i=1}^\infty N(P_i))$\\
Then $h_k(z) = z$ for all $k \geq 1$. So define $G(z_i) = z$ where $z_i = z$ for all $i \geq 0.$\\
{\bf Case 2:} $z \in N(P_{i_0})$ for some $i_0$\\
Then there exists a unique point $(\alpha_i) \in \text{Lim}(M)$ such that $\alpha_j = z$ for some $j$. In fact, there is a unique pair of a point $w \in \SS^2(Q_2)$ and  $j \geq 1$ such that $(\alpha_i) \in \text{Lim}(M)$ where $\alpha_i = w$ for all $i \leq j$ and $\alpha_i = z$ for all $i > j.$ Define $G(\alpha_i) = z$.
Lastly, define $G(z_i) = V$ where $z_i = V$ for all $i \geq 0$ and $V = [(0,0)].$ 
Then $G$ is a surjective function from $\text{Lim}(M)$ to $\SS^2(Q_2).$\\
Let $k \geq 1.$ Let $(z_i) \in \text{Lim}(M).$ Notice one important characteristic of an element $(z_i)$ of $\text{Lim}(M)$: it is eventually constant in the sense that either there exists $w \in \SS^2(Q_2)$ such that $z_i = w$ for all $i \geq 0$, or there exists a unique triple $(v, w, j)$ of distinct $v, w \in \SS^2(Q_2)$ and $j \geq 1$ such that $z_i = v$ for all $i < j$ and $z_i = w$ for all $i \ge j.$ For the former case, $\phi_k \circ G (z_i) = \phi_k(w) = w = \psi_k(z_i).$
For the latter case, $$\phi_k \circ G(z_i) = \phi_k(w) = 
\begin{cases}
v &\text{for} \ k < j,\\
w &\text{for} \ k \geq j.
\end{cases}
$$
So we conclude that $\phi_k \circ G = \psi_k$ for all $k \in \NN.$
Observe the facts that  $$H_1 : \SS^2(Q_2) \rightarrow \SS^2(Q_2) \ \text{such that} \ \phi_k \circ H_1 = \phi_k \ \forall k \in \NN \rightarrow H_1 = \text{id}_{\SS^2(Q_2)} \ \text{and}$$
$$H_2 : \text{Lim}(M) \rightarrow \text{Lim}(M) \ \text{such that} \ \psi_k \circ H_2 = \psi_k \ \forall k \in \NN \rightarrow H_2 = \text{id}_{\text{Lim}(M)}.$$
This gives that $G \circ F = \text{id}_{\SS^2(Q_2)}$ and $F \circ G = \text{id}_{\text{Lim}(M)}.$ So $F : \SS^2(Q_2) \rightarrow \text{Lim}(M)$ is a bijective continuous function.
Since $\SS^2(Q_2)$ is compact, $F$ is a closed map. Hence, $F$ is a homeomorphism.
\end{proof}
%--------------------------------
Now, to see that $Q_n \cong \SS^2$ for $n \geq 3$, one first note that the length of $J_1$ and $J_1'$ exceed $\frac{1}{2}$ so there are extra side identifications on $Q_n$ which cannot be seen on $Q_2$. It can be described via $\tilde J_1', \hat J_1', \check J_1'$ and $\tilde J_1, \hat J_1, \check J_1^*:$
$$\tilde{J_1'} = J_1' \cap \Big(\Big[0, \frac{n-2}{2n}\Big] \times \{1\}\Big), \tilde{J_1} = R(\tilde{J_1'}), \hat{J_1'} = J_1' \cap \Big(\Big(\frac{n-2}{2n}, \frac{1}{2}\Big] \times \{1\}\Big), \hat{J_1} = R(\hat J_1') \ \text{and}$$
$$\check J_1' = J_1' \setminus (\tilde J_1' \cup \hat J_1'), \check J_1 = J_1 \setminus (\tilde J_1 \cup \hat J_1) \ \text{which give} \ J_1 = \tilde J_1 \cup \hat J_1 \cup \check J_1, J_1' = \tilde J_1' \cup \hat J_1' \cup \check J_1'.$$ 
%------------------------------------------
 \begin{figure}[htp]
    \centering
    \includegraphics[width=5cm]{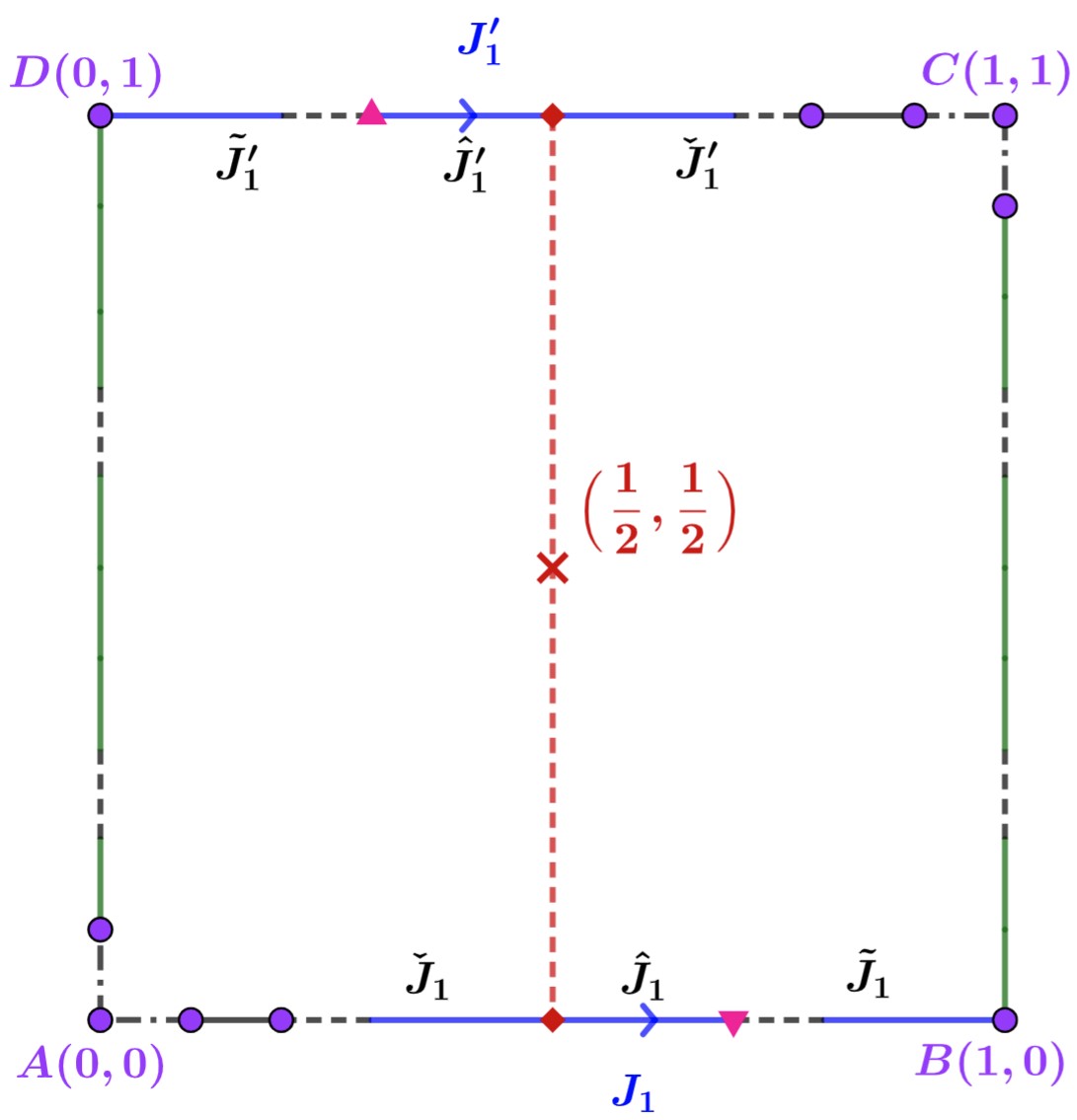}
    \caption{Decomposition of $J_1$ and $J_1'$ as $\tilde J_1, \hat J_1, \check J_1$ and $\tilde J_1', \hat J_1', \check J_1'$, respectively\,.}
    \end{figure}
%------------------------------------
One see that $Q_n$ can be described by identifying each part on the boundary of $[0,0.5] \times [0,1]$ symmetrically to its middle point. There is an identification on $\tilde J_1'$ with $\check J_1$ which does not appear on $Q_2$. In fact, one can regard $Q_n$ as $Q_2$ with an additional strip between $\tilde J_1'$ and $\check J_1$ (see the Figure below). 
%------------------------------------------
 \begin{figure}[htp]
    \centering
    \includegraphics[width=6.5cm]{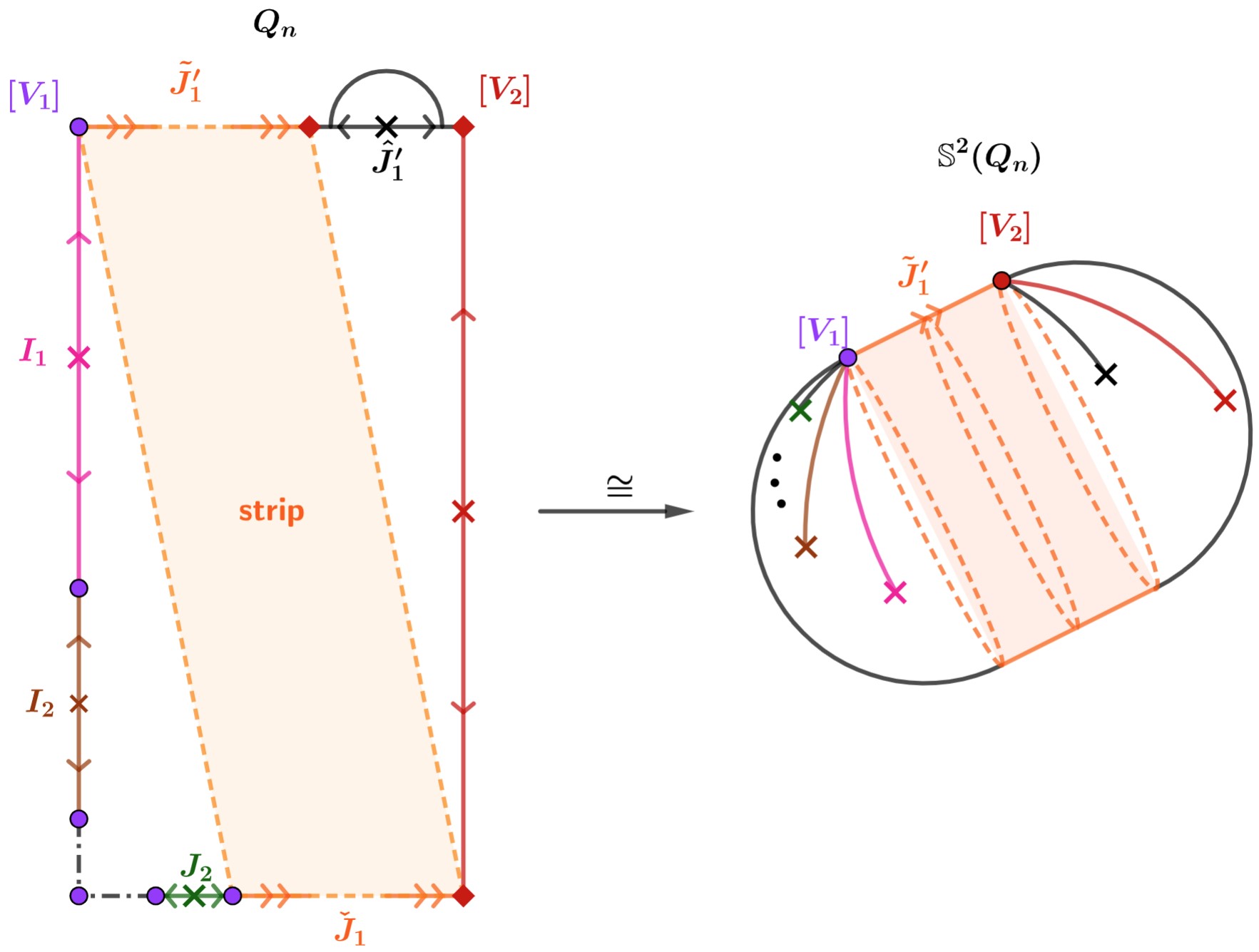}
    \caption{The quotient $Q_n$ illustrated as a quotient of $[0,0.5] \times [0,1]$. It is homeomorphic to $\SS^2(Q_n), n \geq 3$\,.}
\end{figure}\\
%-----------------------------------
We build a corresponding space $\SS^2(Q_n)$ of $Q_n$ the same way of $\SS^2(Q_2)$ and $Q_2.$ Via the same method used for $Q_2$, one then builds an inverse limit system which is homeomorphic to $\SS^2$. Then using the universal property of the inverse limit of topological spaces, $Q_n \cong \SS^2.$ We state this result as a Corollary to Proposition \ref{prop:301}.
%------------------------------------------
\begin{corollary}\label{cor:301}
The space $Q_n$ is homeomorphic to $\SS^2, n \geq 3$
\end{corollary}
%--------------------------------
%%%%%%%%%%%%%%%%%%%%%%%%%%%%%%%%%%%%%%%%%%%%%%%%%%%%%
\subsubsection{The existence of a dynamical system on $Q_n$ induced from the system $(C_n, B_n)$} We aim to verify here the existence of a dynamical system on $Q_n$ induced from $(C_n, B_n), n \geq 2$. However, one realized from the previous subsection that $Q_n, n \geq 3,$ has extra side identifications which do not appear in $Q_2$. These extra side identifications complicate the existence proof of systems on $Q_n$ for $n \geq 3$. Moreover, the proof for the case of $Q_2$ was already discussed in the talk given by Dr. Daniel Mayer (see the acknowledgement). We thus focus on these cases, and refer the case of $Q_2$ to the author dissertation (see \cite{WH}).

Let $n \geq 3.$ Each class $\omega = [(x,y)] = [(0.x_1x_2\ldots_n, 0.y_1y_2\ldots_n)]$ under $\sim_R$ can be written explicitly as:
$$\omega = \begin{cases}
\{\overline{(x,y)}, \overline{R(x,y)}\} &\mbox{if} \ (x, y) \in (0,1)^2,\\
\{(0, 0.y_1y_2y_3\ldots_n), (0, 0.\bar{y_1}^*y_2^*y_3^*\ldots_n)\} &\mbox{if} \ (x,y) \in I_1, \\
\{(0, 0.\bar{n}^*\bar{n}^*\ldots\bar{n}^*y_k^*y_{k+1}^*\ldots_n), (0, 0.0\ldots0\hat{y_k}y_{k+1}\ldots_n)\} &\mbox{if} \ (x,y) \in I_k, k \geq 2 \\
\{(0.x_1x_2x_3..._n, 0), (0.\hat{x_1}^*x_2^*x_3^*..._n, 0)\} &\mbox{if} \ (x,y) \in \tilde{J_1} \cup \tilde{J_1'},\\
\{(0.x_1^*x_2^*x_3^*\ldots_n, 1), (0.\bar{x_1}x_2x_3\ldots_n, 1)\} &\mbox{if} \ (x,y) \in \hat{J_1'}\\
\{(0.0\ldots0x_k^*x_{k+1}^*\ldots_n, 0), (0.0\ldots0\hat{x_k}x_{k+1}\ldots_n, 0)\} &\mbox{if} \ (x,y) \in J_k, k \geq 2
\end{cases}$$ where $ a^- = \bar a := a - 1, a^+ = \hat a := a+1, a^* := (n-1) - a$ and 
$[(0,0)] = \{\overline{(0,0)}\}.$
%----------------------------------------
 \begin{figure}[htp]
    \centering
    \includegraphics[width=4cm]{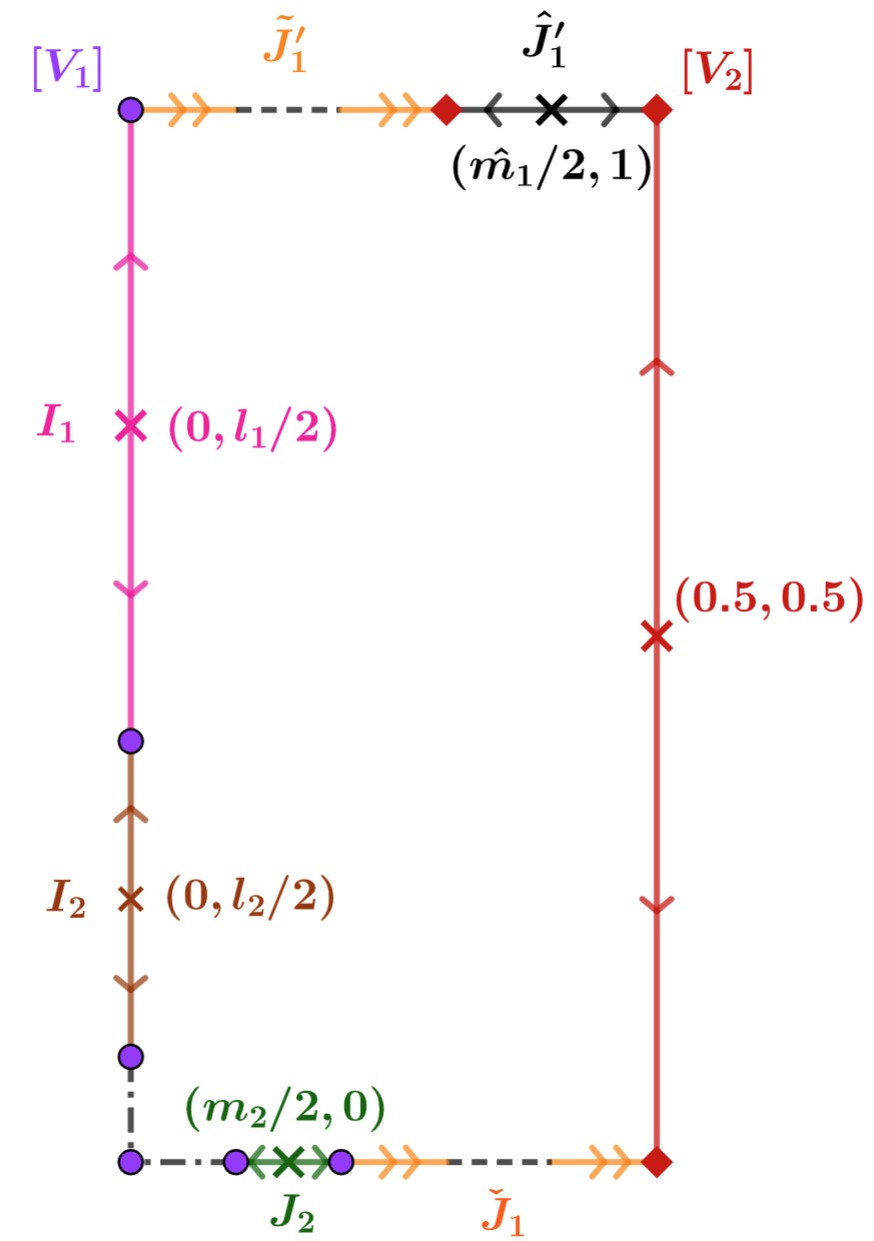}
    \caption{The identification on the boundary of $[0, 0.5] \times [0,1]$ assigned by $\sim_R$\,.}
    \end{figure}\\
%-------------------------------
Observe that each $\hat J_1', I_{k-1}$ and $J_k$ is identified under $\sim_R$ by assigning points corresponding to each middle point $\Big(\frac{\hat m_1}{2}, 1\Big), \Big(0, \frac{l_{k-1}}{2}\Big)$ and $\Big(\frac{m_k}{2}, 0\Big).$ 
%---------------------------------
\begin{proposition}\label{prop:302}
The space $C_n$ is a 2-to-1 branched-covering space of $Q_n$ with countably infinite branched points $\{[(0,0)], [(0.5)], [(\hat{m_1}/2,1)], [0, l_{k-1}/2], [(m_k/2,0)] : k \geq 2\}$. In fact, there exists a homeomorphism $T_n : Q_n \rightarrow Q_n$ such that $\hat P_n \circ B_n = T_n \circ \hat P_n$ where $\hat P_n : C_n \rightarrow Q_n$ is the quotient map.
\end{proposition}
\begin{proof}
Let $z = [(x,y)] = [(0.x_1x_2\ldots_n, 0.y_1y_2\ldots_n)] \in Q_n.$\\
Assume first that $0 < 0.x_1x_2\ldots_n < 0.5, 0 < 0.y_1y_2\ldots_n < 1.$\\
Notice that $(*) \ \ R(\Ecal_i) = \Ecal_{n-1-i} = \Ecal_{i^*} \ \mbox{and} \ R(\Vcal_i) = \Vcal_{n-2-i} = \Vcal_{i^* - 1} = \Vcal_{\bar{i^*}} \ \text{for all} \ i = 0, ..., n-1$  where $\bar{a} = a - 1, a^* = (n-1)-a$ (recall notations $\Ecal_i$, $\Vcal_j$ used in Lemma \ref{lem:301}).\\
Also note that $\hat P_n^{-1}(\{z\}) = \{(0.x_1x_2\ldots_n, 0.y_1y_2\ldots_n), (0.x_1^*x_2^*\ldots_n, 0.y_1^*y_2^*\ldots_n)\}.$\\
Since $0 < 0.x_1x_2\ldots_n < 0.5$, there exists $0 \leq k \leq \lfloor \frac{n-1}{2} \rfloor$ such that $0.x_1x_2\ldots_n \in \Ecal_k \cup \Vcal_k$.\\ This yields that $0.x_1^*x_2^*\ldots_n \in \Ecal_{k^*} \cup \Vcal_{\bar{k^*}}.$ \\
If $(0.x_1x_2\ldots_n, y) \in \Ecal_k,$ then 
\begin{align*}
\hat{P_n}B_n \hat{P_n}^{-1}(\{z\}) &= \hat{P_n}B_n\{(0.kx_2\ldots_n, 0.y_1y_2\ldots_n), (0.k^*x_2^*\ldots_n, 0.y_1^*y_2^*\ldots_n)\} \\
 &= \hat{P_n}\{(0.x_2x_3\ldots_n, 0.ky_1y_2\ldots_n), (0.x_2^*x_3^*\ldots_n, 0.k^*y_1^*y_2^*\ldots_n)\} \\
 &= \{[(0.x_2x_3\ldots_n, 0.ky_1y_2\ldots_n)], [(0.x_2^*x_3^*\ldots_n, 0.k^*y_1^*y_2^*\ldots_n)]\} \\
 &= \{[(0.x_2x_3\ldots_n, 0.ky_1y_2\ldots_n)]\}
\end{align*}
because $[(0.x_2x_3\ldots_n, 0.ky_1y_2\ldots_n)] = [(0.x_2^*x_3^*\ldots_n, 0.k^*y_1^*y_2^*\ldots_n)]$.\\
If $(0.x_1x_2..._n, y) = (\frac{k+1}{n}, y) \in \Vcal_k,$ then
\begin{align*}
\hat{P_n}B_n \hat{P_n}^{-1}(\{z\}) &= \hat{P_n}B_n\Big\{\Big(\frac{k+1}{n}, y\Big), \Big(\frac{n-k-1}{n}, 1 - y\Big)\Big\} \\
 &= \hat{P_n}\Big\{\Big(0, \frac{y+k+1}{n}\Big), \Big(0, \frac{n-k-y}{n}\Big)\Big\} \\
 &= \Big\{\Big[\Big(0, \frac{y+k+1}{n}\Big)\Big], \Big[\Big(0, \frac{n-k-y}{n}\Big)\Big]\Big\}.
\end{align*}
Observe that $\frac{y+k+1}{n} + \frac{n-k-y}{n} = 1+\frac{1}{n} = l_1$ where $(0, l_1/2)$ is the middle point of $I_1$ yielding that $$\Big[\Big(0, \frac{y+k+1}{n}\Big)\Big] = \Big[\Big(0, \frac{n-k-y}{n}\Big)\Big].$$
Now assume that $x = 0$ and $y = 0.y_1y_2..._n \in (0,1)$ such that $z = [(0,y)] \neq [(0,0)]$.\\
Then $(0,y) \in I_k = \{0\} \times \Big(\frac{1}{n^{k}}, \frac{1}{n^{k-1}}\Big)$ for some $k \in \NN.$ \\
Note that
\begin{align*}
 \hat{P_n}B_n\hat{P_n}^{-1}(\{z\})  &= \hat{P_n}B_n\{(0, 0.\bar{n}^*\bar{n}^*\ldots\bar{n}^*y_k^*y_{k+1}^*\ldots_n), [(0, 0.0\ldots0y_k^+y_{k+1}\ldots_n)\}\\   
 &= \hat{P_n}\Big\{(0,0.0\bar{n}^*\bar{n}^*\ldots\bar{n}^*y_k^*y_{k+1}^*\ldots_n), \Big(0, \frac{0.0\ldots0y_k^+y_{k+1}\ldots_n}{n}\Big)\Big\}\\
 &= \{[(0,0.0\bar{n}^*\bar{n}^*\ldots\bar{n}^*y_k^*y_{k+1}^*\ldots_n)]\}
\end{align*}
because $\frac{0.\bar{n}^*\bar{n}^*\ldots\bar{n}^*y_k^*y_{k+1}^*\ldots_n}{n} + \frac{0.0\ldots0\hat{y_k}y_{k+1}\ldots_n}{n} = \frac{l_k}{n} = l_{k+1},$ $(0, l_{k+1}/2)$ is the middle point of $I_{k+1}$.\\
Next, let $z = [(x,0)]$ be such that $(x,0) \in J_1 $ and $x \leq 0.5$ (i.e. $(x,0) \in \check J_1$). \\
Observe the following facts $$\frac{k}{n} < x < \frac{k+1}{n} \rightarrow \frac{n-k}{n} < 1+\frac{1}{n} - x < \frac{n-k+1}{n}.$$
This yields that
\begin{equation*}
\begin{split}
\hat{P_n}B_n\hat{P_n}^{-1}(\{z\}) &= \hat{P_n}B_n\Big\{(x, 0), \Big(1 + \frac{1}{n} - x, 0\Big)\Big\} \\
&= \hat{P_n}\Big\{\Big(nx-k, \frac{k}{n}\Big), \Big(n(1+\frac{1}{n} - x) - (n-k), \frac{n-k}{n}\Big)\Big\} \\
&= \hat{P_n}\Big\{\Big(nx-k, \frac{k}{n}\Big), \Big(1+k - nx, \frac{n-k}{n}\Big)\Big\} \\
%&= \Big\{\Big[\Big(nx-k, \frac{k}{n}\Big)\Big]\Big\}
\end{split}
\end{equation*}
because $R(nx-k, \frac{k}{n}) = (1+k-nx, \frac{n-k}{n}).$\\
Next, let $z = [(x,y)]$ with $(x,y) \in \hat{J_1}$. \\
It is necessary to consider two cases.\\
{\bf Case 1:} $n$ is even\\
Then $z = [(x,y)] = \{(0.x_1x_2x_3\ldots_n,0), (0.\bar{x_1}^*x_2^*x_3^*\ldots_n,0)\}$ where $x_1 = n/2.$
\begin{figure}[htp]
    \centering
    \includegraphics[width=9cm]{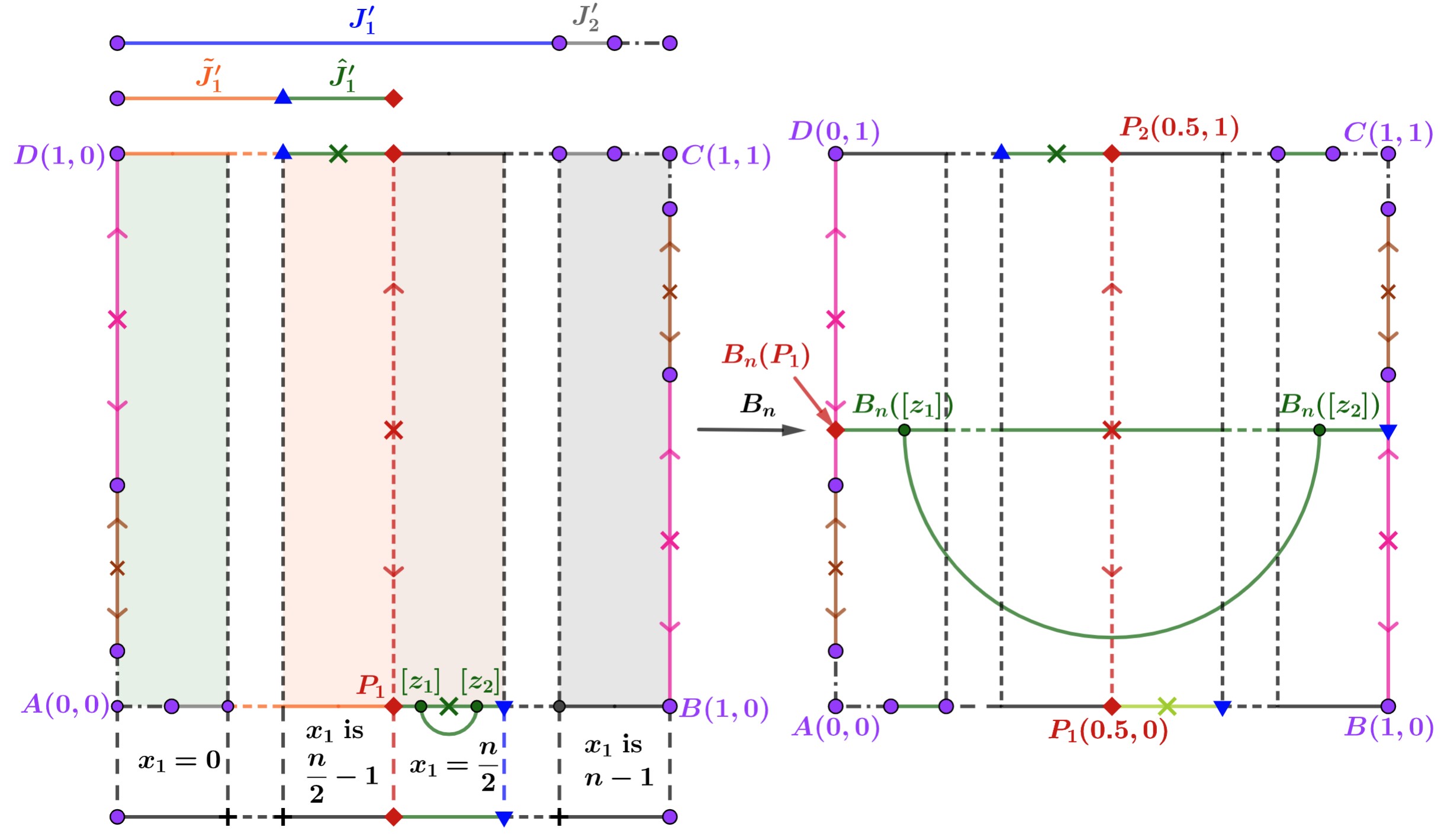}
    \caption{Images of $z_1, z_2$ under $B_n$ where $z_1 = (0.x_1x_2\ldots_n, 0), z_2 = (0.\bar{x_1}^*x_2^*\ldots_n, 0)$\,.}
    \label{fig10}
\end{figure} \\
Hence
\begin{align*}
\hat{P_n}B_n\hat{P_n}^{-1}(\{z\}) &= \hat{P_n}B_n\Big\{\Big(\frac{n}{2n} + \sum_{j=2}^\infty \frac{x_j}{n^j},0\Big), \Big(\Big(\frac{n-1}{n} - \frac{n/2-1}{n}\Big) + \sum_{j=2}^\infty \frac{n-1-x_j}{n^j}, 0\Big)\Big\} \\
&= \hat{P}_n\{(0.x_2x_3\ldots_n, 0.5), (0.x_2^*x_3^*\ldots_n, 0.5)\}\\
&= \{[(0.x_2x_3\ldots_n, 0.5)], [(0.x_2^*x_3^*\ldots_n, 0.5)]\}.
\end{align*}
with the fact that $\sum_{j=1}^\infty \frac{x_{j+1}}{n^j} + \sum_{j=1}^\infty \frac{n-1-x_{j+1}}{n^j} = \sum_{j=1}^\infty \frac{n-1}{n^j} = 1$ which yields that $$R(0.x_2x_3\ldots_n, 0.5) =(0.x_2^*x_3^*\ldots_n, 0.5).$$
{\bf Case 2:} $n$ is odd \\
Then $z = (x,y) = \{(0.x_1x_2x_3\ldots_n,0), (0.\bar{x_1}^*x_2^*x_3^*\ldots_n,0)\}$ where $x_1 = \frac{n-1}{2}.$
\begin{figure}[htp]
    \centering
    \includegraphics[width=9cm]{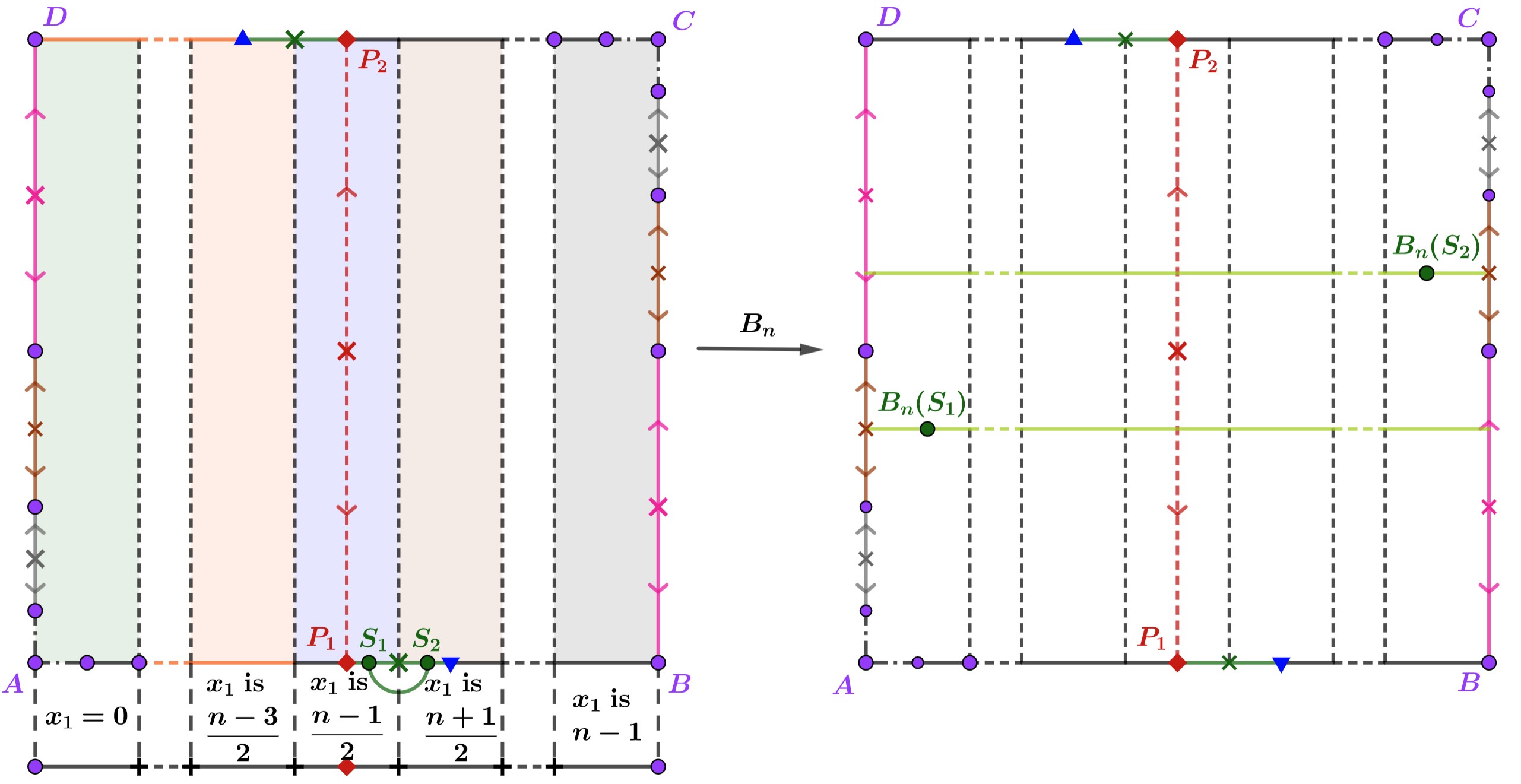}
    \caption{Images of $S_1, S_2$ under $B_n$ where $S_1 = (0.x_1x_2\ldots_n, 0), S_2 = (0.\bar{x_1}^*x_2^*\ldots_n, 0)$\,.}
    \label{fig11}
\end{figure} \\
Therefore
\begin{align*}
\hat{P_n}B_n\hat{P_n}^{-1}(\{z\}) &= \hat{P_n}B_n\Big\{\Big(\frac{n-1}{2n} + \sum_{j=2}^\infty \frac{x_j}{n^j},0\Big), \Big(\Big(\frac{n-1}{n} - \frac{\frac{n-1}{2}-1}{n}\Big) + \sum_{j=2}^\infty \frac{n-1-x_j}{n^j}, 0\Big)\Big\} \\
&= \hat{P_n}B_n\Big\{\Big(\frac{n-1}{2n} + \sum_{j=2}^\infty \frac{x_j}{n^j},0\Big), \Big( \frac{n+1}{2n} + \sum_{j=2}^\infty \frac{n-1-x_j}{n^j}, 0\Big)\Big\} \\
&= \hat{P}_n\Big\{\Big(0.x_2x_3\ldots_n, \frac{n-1}{2n}\Big), \Big(0.x_2^*x_3^*\ldots_n, \frac{n+1}{2n}\Big)\Big\}\\
&= \Big\{\Big[\Big(0.x_2x_3\ldots_n, \frac{n-1}{2n}\Big)\Big]\Big\}
\end{align*}
 because $\Big[\Big(0.x_2x_3\ldots_n, \frac{n-1}{2n}\Big)\Big]= \Big[R\Big(0.x_2x_3\ldots_n, \frac{n-1}{2n}\Big)\Big] = \Big[\Big(0.x_2^*x_3^*\ldots_n, \frac{n+1}{2n}\Big)\Big].$\\
Lastly, let $z = [(x,0)]$ where $(x,0) \in J_k$ for some $k \geq 2.$\\
In this case, $z = \{(0.0\ldots0x_k^*x_{k+1}^*\ldots_n,0), (0.0\ldots0x_k^+x_{k+1}\ldots_n,0)\}$. \\
Hence, 
\begin{align*}
    \hat{P_n}B_n\hat{P_n}^{-1}(\{z\}) &= \hat{P_n}B_n\{(0.0\ldots0x_k^*x_{k+1}^*\ldots_n,0), (0.0\ldots0x_k^+x_{k+1}\ldots_n,0)\} \\
    &= \hat{P_n}\Big\{\Big(\frac{0.x_k^*x_{k+1}^*\ldots_n}{n^{k-2}},0\Big), \Big(\frac{0.x_k^+x_{k+1\ldots_n}}{n^{k-2}},0\Big)\Big\} \\
    &=\Big\{\Big[\Big(\frac{0.x_k^*x_{k+1}^*\ldots_n}{n^{k-2}},0\Big)\Big]\Big\}
\end{align*}
because $\frac{0.x_k^*x_{k+1}^*\ldots_n}{n^{k-2}} + \frac{0.x_k^+x_{k+1}\ldots_n}{n^{k-2}} = \frac{1+ \frac{1}{n}}{n^{k-2}} = \frac{1}{n^{k-2}}+\frac{1}{n^{k-1}} = m_{k-1}$ for $k \geq 2$.\\
Since $B_n$ is a homeomorphism, we can conclude that there exists a homeomorphism \\$T_n : Q_n \rightarrow Q_n$ such that $\hat{P_n} \circ B_n = T_n \circ \hat{P_n}.$
\end{proof}
%--------------------------------------------
As a summary, we have that Proposition \ref{prop:301}, Corollary \ref{cor:301} together with Proposition \ref{prop:302} give Theorem \ref{thmx:00A}. Then the fact regarding density of periodic points stated in Lemma \ref{lem:301} yields Theorem \ref{thmx:00B}.
%%%%%%%%%%%%%%%%%%%%%%%%%%%%%%%%%%%%%%%%%%%%%%%%%%%%
\section{Metric entropy values of induced dynamical systems on the carpet}
This section starts by listing steps we used to extend a given pair $(Y, S)$ to a quadruple $(Y, S, \Bcal(Y), \mu_Y)$. A more detailed discussion can be found in \cite{BO}. One notes that this method works regardless of the initial dynamical system $(X, T)$. \\
{\bf Step 0:} In some cases, there is a pre-initial system $(\tilde X, \tilde T, \Bcal(\tilde X), \tilde \mu)$ which provides the initial system $(X, T)$ a measure-theoretic structure via Lemma \ref{lem:201}. If an initial system $(X, T)$ has a natural measure-theoretic structure, one skips this step.\\
{\bf Step 1:} Assume that the initial system $(X, T)$ has a measure-theoretic structure $(X, T, \Bcal(X), \mu_X)$. A system on the quotient sphere $(\Scal_0, H_0)$ extends to $(\Scal_0, H_0, \Bcal(\Scal_0), \nu)$ through Lemma \ref{lem:201}.\\
{\bf Step 2:} There is the $0-$th coordinate projection $\Pi_0 : S_\infty(T) \rightarrow \Scal_0, (z_i)_{i \geq 0} \mapsto z_0$. Then a measure $\mu(A) := \nu(\Pi_0(A \cap M_\infty)) $ is a probability measure on $S_\infty(T),$ where $M_\infty = \{(u_i) \in S_\infty : U_0 \in \Pi_0(S_\infty(T)) \setminus \Ocal\}.$ The quadruple $(S_\infty(T), H_\infty(T), \Bcal(S_\infty(T)), \mu)$ is a dynamical system which is isomorphic to $(\Scal_0, H_0, \Bcal(\Scal_0), \nu)$. Since these two families have finite numbers of branch points 
\vspace{0.4cm}

We now briefly discuss entropy values of $H_\infty(F_A)$ and $H_\infty(T_\lambda)$. Both families $(\TT^2, F_A)$ and $(X_g, T_\lambda)$ have their natural measure-theoretic structures induced from $\RR^2$. It is well-known that their entropy with respect to their invariant measures are $\log(|\lambda'|)$ and $\log(\lambda)$ where $\lambda'$ is the leading eigenvalue of $F_A.$ Since these two families have a finite numbers of branch points, as pointed out earlier in the discussion preceding Proposition \ref{prop:202}, entropy values of systems on the quotient sphere are $\log(|\lambda'|)$ and $\log(\lambda)$, respectively. Since each system on the carpet is isomorphic to its base system of the quotient sphere, $\log(|\lambda'|)$ and $\log(\lambda)$ are respectively the entropy values of $H_\infty(F_A)$ and $H_\infty(T_\lambda)$. One notes that the crucial step is that Proposition \ref{prop:203} can be applied immediately to show that entropy values of $F_A$ and $T_\lambda$ are lower bounds of entropy values of homeomorphisms on the quotient spheres. This is due to the fact that they have finitely many branch points. This situation is slightly different for cases of $(C_n, B_n)$ for $n \geq 2$, for which we now provide a proof.
%-------------------------------------------------
\begin{proposition}\label{prop:401}
Let $n \geq 2$. For a fixed probability vector $P = (p_0, p_1, \ldots, p_{n-1})$, there is a measure-theoretic dynamical system on $(S_\infty, H_\infty)$ with an entropy value of $-\sum_{i=0}^{n-1} p_i\log(p_i).$ As a consequence, every positive real number is realized as an entropy value of dynamical systems on $S_\infty.$
\end{proposition}
\begin{proof}
One refers to \cite{KH} for the following facts: for each $(\Sigma_n, \sigma_n)$ and for each probability vector $P = (p_0, p_1, \ldots, p_{n-1})$ (i.e. $0 \leq p_i < 1$ and $\sum_{i=0}^{n-1} p_i = 1$), there is a probability measure $\rho_P$ defined on $\Bcal(\Sigma_n)$ such that $(\Sigma_n, \sigma_n, \Bcal(\Sigma_n), \rho_P)$ is a dynamical system. Moreover, $h_{\rho_P}(\sigma_n) = - \sum_{i=0}^{n-1} p_i \log(p_i)$. Note that Lemma \ref{lem:301} and Proposition \ref{prop:302} yield that $(Q_n, T_n)$ is a factor of $(\Sigma_n, \sigma_n)$. So we have $$h_\nu(T_n) \leq h_{\rho_P}(\sigma_n) = - \sum_{i=0}^{n-1} p_i \log(p_i) \ \text{where} \ \nu \ \text{is the push-forward measure}.$$
To obtain the reverse inequality, we note that there is a unique point with countably infinite preimage on $\Sigma_n$, and all other points have finite preimage. Since the measure $\rho_P$ assigns a measure zero to any countably infinite set, using Proposition \ref{prop:203}, we obtain the reverse inequality and get that $h_{\nu}(T_n) = - \sum_{i=0}^{n-1} p_i\log(p_i).$
Observe that $$\lim_{N \rightarrow \infty} \Big(-\frac{N-1}{N}\log\Big(\frac{N-1}{N}\Big) + \frac{\log(N)}{N}\Big) = 0 \ \text{and} \ \sum_{i=0}^{N-1} \frac{\log(N)}{N} = \log(N) \xrightarrow{N \rightarrow \infty} \infty. $$ By continuity and connectedness, for all $N \geq 2,$ $\Big[-\frac{N-1}{N} \log\Big(\frac{N-1}{N}\Big) + \frac{\log(N)}{N}, \log(N)\Big] \subseteq E(N)$ where $E(N) := \{-\sum_{i=0}^{N-1} p_i\log(p_i) \mid P = (p_0, p_1, \ldots, p_{N_1}) \ \text{is a probability vector}\} \subseteq \RR$.
So $(0, \infty) \subseteq \cup_{N \geq 2} E(N).$
\end{proof}
%-------------------------------------------------
%%%%%%%%%%%%%%%%%%%%%%%%%%%%%%%%%%%%%%%%%%%%%%%%%%%
\section{The failure of Bowen specification property}
We present in this section a proof that the systems on the carpet do not have the approximate product property (and hence they cannot have the Bowen specification). The proof here is a modification of the proof found in \cite{BO}. We discuss two key facts prior to giving the proof: \textit{the local behavior of points near a hyperbolic fixed point} and \textit{the topology of the inverse limit}. 
%%%%%%%%%%%%%%%%%%%%%%%%%%%%%%%%%%%%%%%%%%%%%%%%%%%
\subsection{The local behavior of orbits of points near a hyperbolic fixed point} Observe first that any fixed point of $F_A^k, T_\lambda^k$ and $B_n^k$ is hyperbolic for any $k \in \NN$. Hence the behavior of orbits of points under iterations of $F_A, T_\lambda$ or $B_n$ in small neighborhoods of their hyperbolic fixed points are similar. In particular, the local behavior is analogous to the behavior of points under iterations of a hyperbolic linear map $L_M : \RR^2 \rightarrow \RR^2$ in a neighborhood of the origin.

Note that our definition of a hyperbolic linear map $L_M$ on $\RR^2$ is slightly different from the usual one.
%------------------------------------------
\begin{definition}\label{def:501}
A linear map $L_M : \RR^2 \rightarrow \RR^2$ is called {\bf a hyperbolic linear map} if $L_M$ is a linear transformation of the form
$L_M(x,y)^t = M(x,y)^t$ where $M$ is a $2 \times 2$ matrix with real entries such that both of its eigenvalues do not have modulus $1$ and $|\text{det}(M)|=1.$
\end{definition}
%-------------------------------------------
Using diagonalization coordinates, let $L_M$ be a hyperbolic linear map of the form $L_M(x,y) = (\lambda^{-1}x, \lambda y)$ where $0 < \lambda^{-1} < \lambda.$
Then Proposition \ref{prop:501} describes an iteration pattern of a point in a neighborhood of the origin under $L_M$ in a form of \textit{the proportion of visiting a fixed region of interest}. In particular, we choose a sufficiently small open square $D = (-\lambda\epsilon, \lambda\epsilon)^2 \subseteq \RR^2$ and a specific \textit{region of interest} $D^* \subseteq D$. Then, roughly, a consequence is that the proportion of the number of iterations of a point falling in $D^*$ over the total number of iterations of the point falling in $D$ cannot exceed $\frac{1}{2}.$ Observe that Proposition \ref{prop:501} is a formalization of the discussion stated in \cite{BO}.
%---------------------------------------------
\begin{proposition}\label{prop:501}
Let $E = ((-\lambda\epsilon, \lambda\epsilon) \times ([\epsilon, \lambda\epsilon) \cup (-\lambda\epsilon, -\epsilon])) \cup (([\epsilon, \lambda\epsilon) \cup (-\lambda\epsilon, -\epsilon])) \times (-\lambda\epsilon, \lambda\epsilon) ).$ Let $z \in \RR^2$ be such that there exist positive integers $N, K$ satisfying 
\begin{itemize}
    \item $L_M^N(z) \notin D, L_M^{N+K+1}(z) \notin D$ and $L_M^{N+1}(z) \in D_1^* := D_1 \setminus (E \cup (\RR,0))$ (see Figure $12$), 
    \item for all $j = 1, 2, \ldots, K, L_M^{N+j}(z) \in D,$
\end{itemize}
then $$n_N^K(z) := \frac{|\{j = 1, 2, \ldots, K : L_M^{N+j}(z) \in D_1^*\}|}{K} \leq \frac{1}{2}.$$ 
%-----------------------------------
\begin{figure}[htp]
    \centering
    \includegraphics[width=5cm]{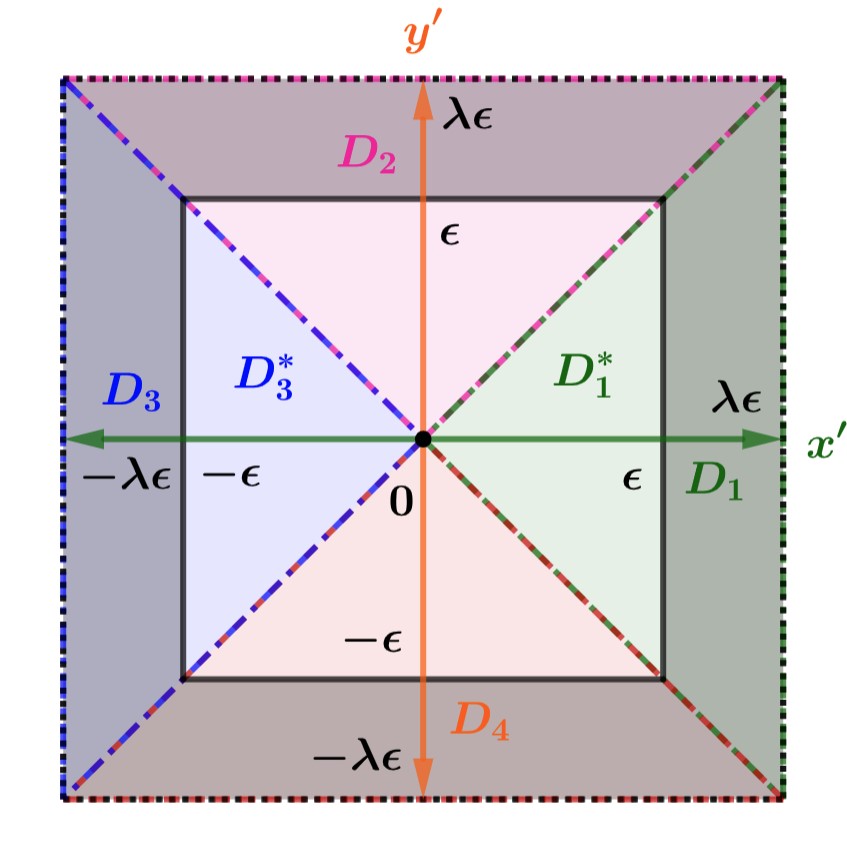}
\caption{Subregions $D_1, D_2, D_3, D_4$ and $D_1^*$ inside $D = (-\lambda\epsilon,\lambda\epsilon)^2$\,.}
\end{figure}
%----------------------------------------
\end{proposition}
\begin{proof} For simplicity, set $\beta = \lambda^{-1}.$ Assume that $z \in \RR^2$ and $N, K \in \NN$ such that $$L_M^N(z) \notin D, L_M^{N+i}(z) \in D\ \forall 1 \leq i \leq K, L_M^{N+K+1}(z) \notin D, L_M^{N+1}(z) = (p,q) \ \mbox{with} \ |p| > |q| >0.$$
Assume that $(p,q) \in D_1^*.$
Then there exists an integer $m$ such that $0 \leq m \leq K$, $$\beta^m|p| > \lambda^m|q|, \beta^{m+1}|p| \leq \lambda^{m+1}|q| < \lambda\epsilon.$$
We also know that $\lambda^{m+1}|q| < \lambda\epsilon$ since  
$
    \lambda\epsilon \leq \lambda^{m+1}|q| \rightarrow \epsilon \leq \lambda^m|q| < \beta^m|p| \leq |p| < \lambda\epsilon\ \rightarrow  (p,q) \in E.
    $
Next, $\lambda^{2m}|q| < \lambda\epsilon$ because $\lambda^{2m}|q| \geq \lambda\epsilon$ implies that $\lambda\epsilon \leq (\beta\lambda)^m \beta^{-m}\lambda^m|q| < \beta^{-m}\beta^m|p| = |p|$ contradicting $(p,q) \in D.$ This in fact yields that $L_M^{N+i}(z) \in D_1^*, L_M^{N+j}(z) \in D_2 \cup D_4$ for all $i = 1, ..., m, j = m+1, ..., 2m.$ Therefore, $$\frac{|\{1 \leq j \leq K : L_M^{N+j}(z) \in D_1^*\}|}{K} \leq \frac{m}{2m} = \frac{1}{2}.$$
\end{proof}
%----------------------------------------------
\begin{remark}
(1) Proposition \ref{prop:501} holds for the case the region $D_1^*$ is changed to $D_3^* : = D_3 \setminus (E \cup (\RR,0))$ or any other subsets of $D_1^*$ or $D_3^*$. \\
(2) We would like to point out a possible error in the discussion found in \cite{BO} on page $343$. Using the notation in \cite{BO}, for each fixed region $D = [-\epsilon, \epsilon]^2$, there is a point $z = (p, q) \in D$ with $|p| \geq |q|$ such that {\bf the minimal $m$} satisfying $a^m|p| \geq b^m|q|$ and $a^{m+1}|p| < b^{m+1}|q|$ is equal to {\bf zero}. Our use of the region $E$ eliminates this exception.
\end{remark}
%%%%%%%%%%%%%%%%%%%%%%%%%%%%%%%%%%%%%%%%%%%%%%%%%%%%
\subsection{The topology of an inverse limit}
We now realize that all three families of initial systems $(\TT^2, F_A),$ $(X_g, T_\lambda)$ and $(C_n, B_n)$ are such that all fixed points are hyperbolic and the induced dynamical systems on the carpet are topologized by the same topology. So the fact concerning the inverse limit topology in this subsection and the proof of the failure of approximate product property in the next subsection are presented on the pair $(S_\infty, H_\infty)$ without the necessity to specify the initial system. In fact, we refer to the initial system ambiguously as the pair $(X, T)$. 

Recall that we blew up the orbit $O_1$ of the point $z_1$ with the period length $n_1$ to create $(\Scal_1, H_1)$ in the first step. Set $G = H_\infty^{n_1} : S_\infty \rightarrow S_\infty$ and $G_0 = H_0^{n_1} : \Scal_0 \rightarrow \Scal_0.$ Assume that eigenvalues of the differential of $T^{n_1}$ are $0 < \lambda^{-1} < \lambda$. The contracting direction at the fixed point $z_1$ intersects the blown up circle $S(z_1) \subseteq \Scal_1$ in exactly $2$ points, denoted them by $v_1$ and $v_2$. Let $z(i) $ denote the element $(z_1, v_i, v_i, \ldots)$ in $S_\infty$, for $i = 1, 2$. By normalizing the diagonalization coordinate of $T^{n_1}$ at $z_1$, we define a small open square region $D = (-\lambda\epsilon, \lambda \epsilon )^2 \subseteq \Scal_0$ satisfying that it contains no images of branch points. We also define open subregion $D_1, D_2, D_3, D_4$ of $D$ and their corresponding open regions $D_1', D_2', D_3', D_4'$ of $\Scal_1$ such that $\pi_1(D_j') = D_j$. Lastly, we introduce four arcs $C_1, C_2, C_3$ and $C_4$ on the boundary circle $S(z_1)$. All regions $D_1, D_2, D_3, D_4, D_1', D_2', D_3', D_4'$ and arcs $C_1, C_2, C_3, C_3$ are shown in Figure 13. 
%---------------------------------------------
\begin{figure}[htp]
    \centering
    \includegraphics[width=9cm]{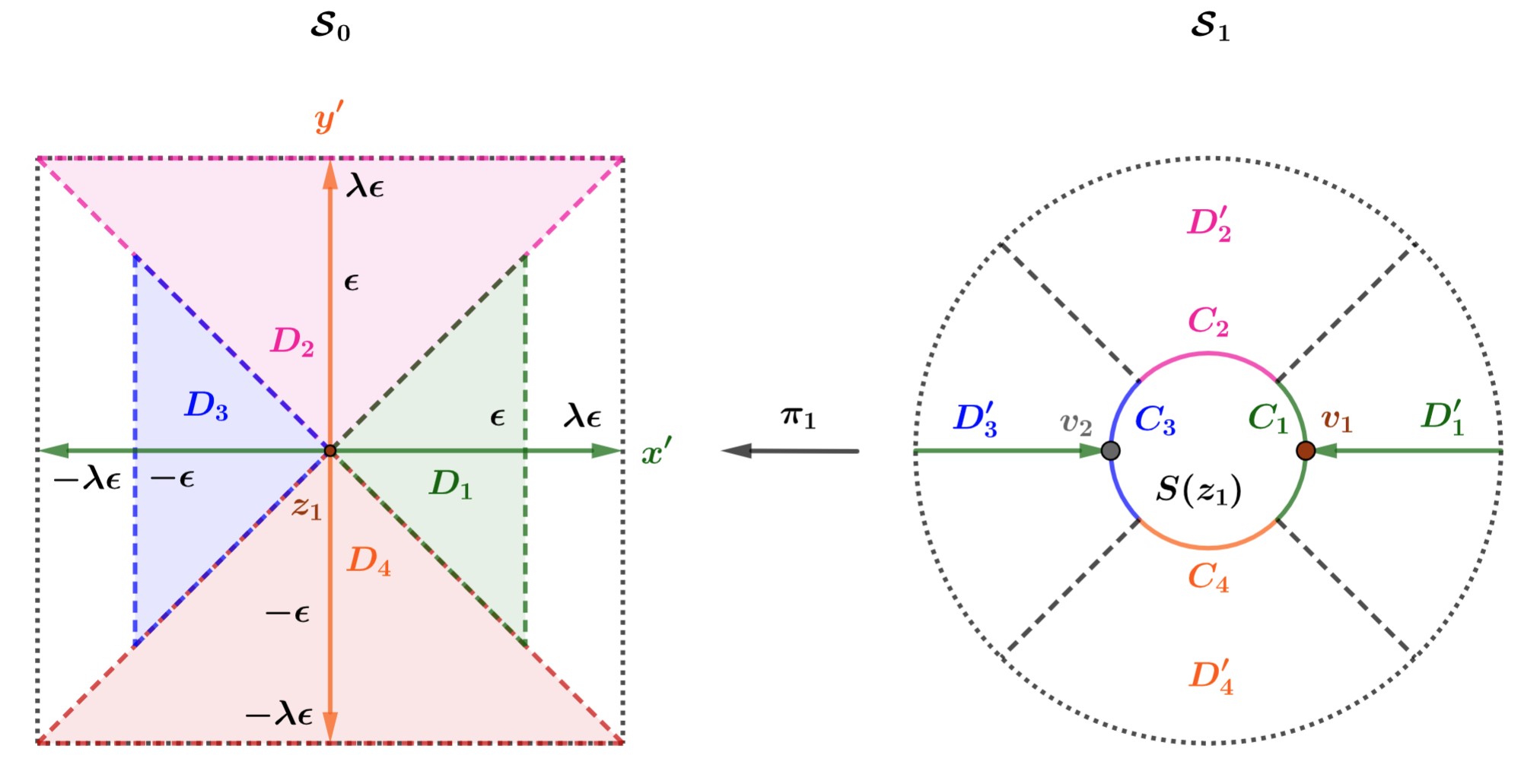}
\caption{Regions $D_1, D_2, D_3, D_4$ in $(-\lambda\epsilon, \lambda\epsilon)^2$, and arcs $C_1, C_2, C_3, C_4$ on the boundary circle of the fixed point $z_1$\,.}
\end{figure}
%-----------------------------------------------

For each $i \geq 0,$ let $\Pi_i : \prod_{j=0}^\infty \Scal_j \rightarrow \Scal_i$ be the $i-$coordinate projection map, $(z_j)_{j \geq 0} \mapsto z_i$.
Note that the product topology on $\prod_{i=0}^\infty \Scal_i$ is metrizable with a metric $d := d_\infty : \prod_{i=0}^\infty \Scal_i \rightarrow [0,\infty)$ defined by $$d((x_i), (y_i)) = \sum_{j=0}^\infty \frac{d_j(x_j,y_j)}{2^j(1+d_j(x_j,y_j))}. $$
The topology on the carpet $S_\infty$ is the subspace topology of $\prod_{i=0}^\infty \Scal_i$.
An open ball centered at $z$ with radius $r > 0$ on $S_\infty$ is denoted by $B_S(z,r) := \{y \in S_\infty : d(z,y) < r\}.$
%--------------------------------------------
\begin{lemma}\label{lem:501} For any sufficiently small $r > 0$, $\Pi_0(B_S(z(1), r) \subseteq D_1.$ That is, the projection of an open ball $B_S(z(1),r)$ to $\Scal_0$ is contained in the triangular region $D_1$ for any small $r > 0$. 
%----------------------------------------
\begin{figure}[htp]
\centering
    \includegraphics[width=9cm]{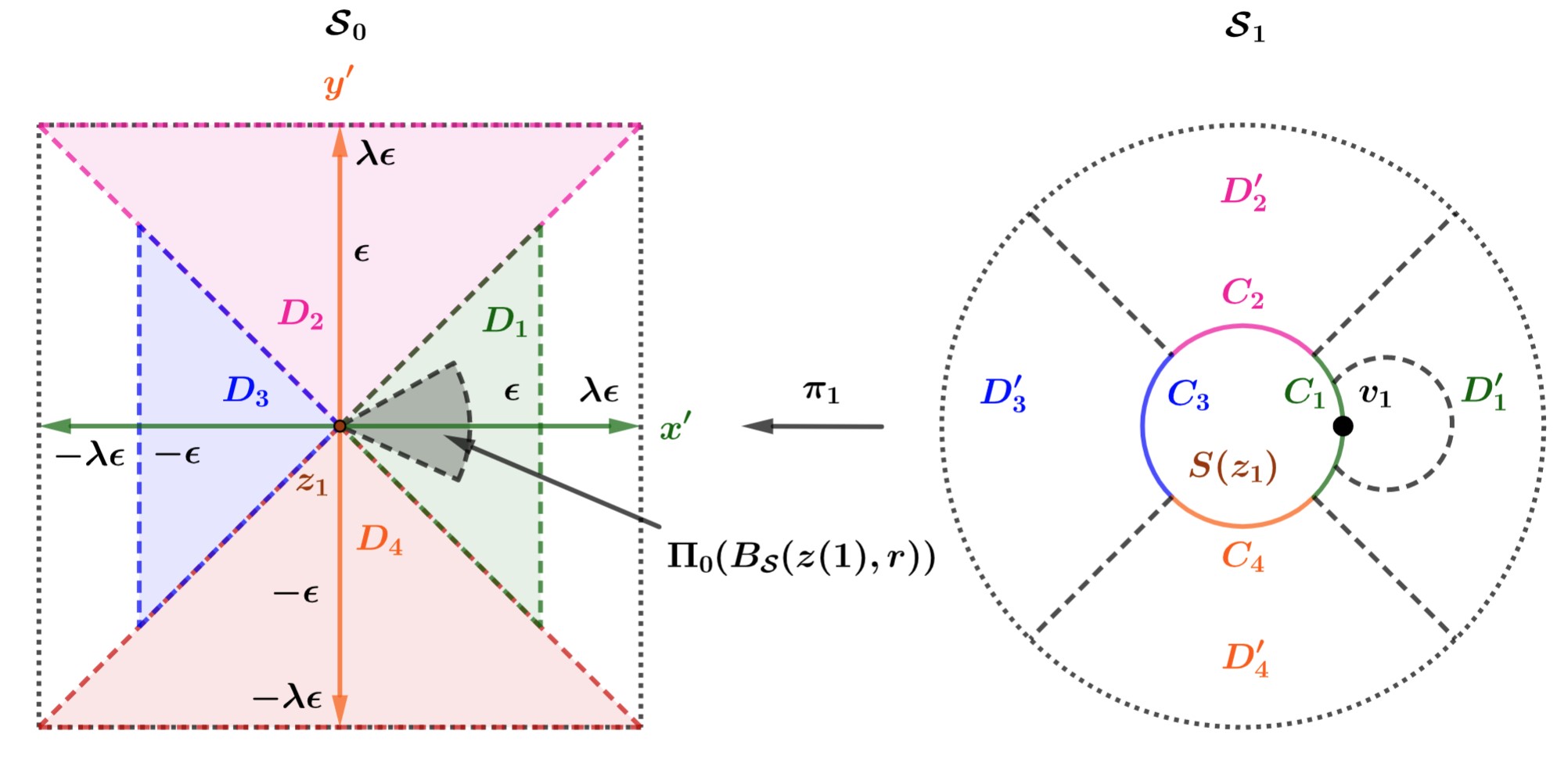}
    \caption{For small  $r > 0,$ the projection to $\Scal_0$ of $B_S(z(1),r)$ is contained in $D_1 \cup \{z_1\}$\,.}
\end{figure}
%----------------------
\end{lemma}
\begin{proof}
Assume that $\sup_{x, y \in \Scal_1} d_1(x,y) < 1.$
Notice that there is $M>0$ such that $$\inf \{d_1(v_1, x) :x \in \cup_{j= 2}^4 (D_j' \cup C_j)\} > M.$$
Let $0 < r < \frac{M}{4}$. Then for any $y = (y_i) \in S_\infty$ with $y_1 \in D_2' \cup D_3' \cup D_4' \cup C_2 \cup C_3 \cup C_4,$ $$d(z(1), y) \geq \frac{d_1(v_1, y_1)}{2(1+d_1(v_1, y_1))} > \frac{M}{4} > r.$$
This yields that $\Pi_0(B_S(z(1),r)) \subseteq D_1 \cup \{z_1\}$ for all sufficiently small $r > 0.$
\end{proof}
%---------------end of lem 5-1-----------------------
\begin{remark}
(1) Lemma \ref{lem:501} holds if $z(1)$ is replaced by $z(2)$. Precisely, for any sufficiently small $r > 0$, $\Pi_0(B_S(z(1),r)) \subseteq D_1 \ \text{and} \ \Pi_0(B_S(z(2)),r) \subseteq D_3.$ \\
(2) Lemma \ref{lem:501} describes that a projection of an open ball of a small radius on the inverse limit $S_\infty(F_A)$ to the $0^{th}-$coordinate is not an open ball centered at $z_1$. The projection is instead a triangular region contained completely in $D_1$. 
\end{remark}
%%%%%%%%%%%%%%%%%%%%%%%%%%%%%%%%%%%%%%%%%%%%%%%%%%%%
\subsection{The proof of the failure of the approximate product property}
%-----------------------------------------------
\begin{theorem}\label{thm:501}
The system $(S_\infty, H_\infty)$ does not have the approximate product property.
\end{theorem}
\begin{proof}
Suppose that $G = H_\infty^{n_1}$ has the approximate product property.\\ Let $\epsilon_0 > 0$ be sufficiently small (Lemma \ref{lem:501}).  
Let $\delta_1 = 10^{-1}$ and $\delta_2 = 10^{-2}.$\\ Let $N = N(\epsilon, \delta_1, \delta_2)$ be the natural number and fix $n \geq \max\{N, 100\}$ with $100|n$. \\Let $(x(i))_{i \in \NN} \subseteq \Scal$ where $x(2i+1) = z(1)$ and $x(2i+2) = z(2)$ for all $i \in \NN_0$. Then there exist $(h_i) \subseteq \NN_0$ satisfying $h_1 = 0, n \le h_{i+1} - h_i \le n + \frac{n}{100}$ and $y = (y_i) \in S_\infty$ such that 
$$|\{0 \leq j < n : d(G^{h_i+j}(y), G^j(x(i))) > \epsilon_0\}| \leq \frac{n}{10}$$ for each $i \in \NN.$ This yields that for each $i \in \NN$, $$G^{h_i+j}(y) \in B_S(G^j(x(i)), \epsilon_0)$$ for at least $\frac{9n}{10}$ times among $h_i, h_i+1, \ldots, h_i+n-1.$ This implies that $$G_0^j(y_0) =: \Pi_0(G^j(y)) \in \Pi_0(B_S(x(1),\epsilon_0)) \subseteq D_1$$ for at least $\frac{9n}{10}$ times out of $0, 1, \ldots, n-1,$ where $G_0 = H_0^{n_1}.$ Then we have that there exist $0 \leq m_1 < m_2$ with $m_2 - m_1 \geq \frac{8n}{10}$ and $ 0 \leq m_3 \leq \frac{n}{10}$ with $m_2+m_3 = n-1$ such that 
$$G^j_0(y_0) \in D_1 \ \text{for all} \ j = m_1, m_1+1, \ldots, m_2.$$
According to Proposition \ref{prop:501}, $G^j_0(y_0) \in D_2$ for all $j = n, n+1, \ldots, n+\frac{7n}{10}.$\\ Since $h_2 < n+ \frac{n}{100}$ and $D_2 \cap D_3 = \emptyset$, it must be that $$|\{0 \leq j < n : d(G^{h_2+j}(y), x(2))\}| > \frac{n}{10}.$$
\end{proof}
%-----------------------------------
\begin{remark}
(1) Theorem \ref{thm:501} holds for a system $(S_\infty, H_\infty)$ with eigenvalues $\lambda < 0$ or $\lambda^{-1} < 0$. In such cases, we consider the system $(S_\infty, H_\infty^{2k})$ which satisfies that $0 < \lambda^{-2k} < \lambda^{2k}.$ \\
(2) As we mentioned earlier, the proof presented in Theorem \ref{thm:501} is valid on $(S_\infty, H_\infty)$ regardless of the initial system being $(\TT^2, F_A), (X_g, T_\lambda)$ or $(C_n, B_n)$.
\end{remark}

\vspace{0.4cm}
 The proof of Theorem \ref{thm:501} gets a contradiction from an argument on the behavior of points in the neighborhood of $z(1)$ and $z(2)$. These two points are representatives of the contracting direction on the boundary circle $S(z_1)$. Though the notion of the Bowen specification is usually defined on a compact metric space, it is an interesting question \textit{if we consider the invariant subset of $S_\infty$ which we exclude all points in $\Ocal$ (so we exclude both $z(1)$ and $z(2)$), does the specification hold on this subset?}
%%%%%%%%%%%%%%%%%%%%%%%%%%%%%%%%%%%%%%%%%%%%%%%%%%%%
\subsection{The invariant subspace $(S^*, H^*)$: does it have the specification ?}
Let $S^*$ be a subset of $S_\infty$ defined by $$S^* = \{(z_i)_{i \geq 0} \in S_\infty : z_0 \in \Scal_0 \setminus \Ocal\}.$$ That is, $S^*$ is a subset of $S_\infty$ where we exclude all blown up boundary circles. Observe that $H(\Ocal) = \Ocal$ so $S^*$ is $H-$invariant. This gives that $(S^*, H^*)$ is a dynamical system where $H^* := H_\infty|_{S^*} : S^* \rightarrow S^*.$ We show that even $(S^*, H^*)$ cannot have the Bowen specification property.

Before we proceed on, we note two facts. First, $S^*$ is not a compact metric space so it is not true that the Bowen specification implies the approximate product property. Second, the map $H^*$ is a restriction of $H_\infty$ so that Proposition \ref{prop:501} still can be applied. Hence, one needs to only investigate if there is a variant of Lemma \ref{lem:501} for $H^*$.
%-------------------------------------------
\begin{lemma}\label{lem:502}
There exists $0 < \delta < \epsilon$ such that for any $Q = (q_i)_{i \geq 0} \in S_\infty$ with $q_0 = z_1 + (x_Q, 0)$ where $ 0 < x_Q < \epsilon - \delta$, $$\Pi_0(B_S(Q, \delta)) \subseteq D_1. $$
\end{lemma}
\begin{proof}
Using the same notation as in the proof of Lemma \ref{lem:501}, there is $M > 0$ such that $$\inf\{d_1(x, y) : \pi_1(x) = z_1 + (s,0) \ \text{with} \ 0 < s < \epsilon, y \in \cup_{j=2}^4 (D_j' \cup C_j) \} > M.$$
We then have that $d(z,Q) > \frac{M}{4}$ for all $z$ with $\Pi_2(z) \in \cup_{j=2}^4 (D_j' \cup C_j)$ and $Q = (q_i)$ with $q_0 = z_1 + (x_Q, 0)$ where $0 < x_Q < \epsilon.$ 
%--------------------------------------
\begin{figure}[htp]
    \centering
    \includegraphics[width=9cm]{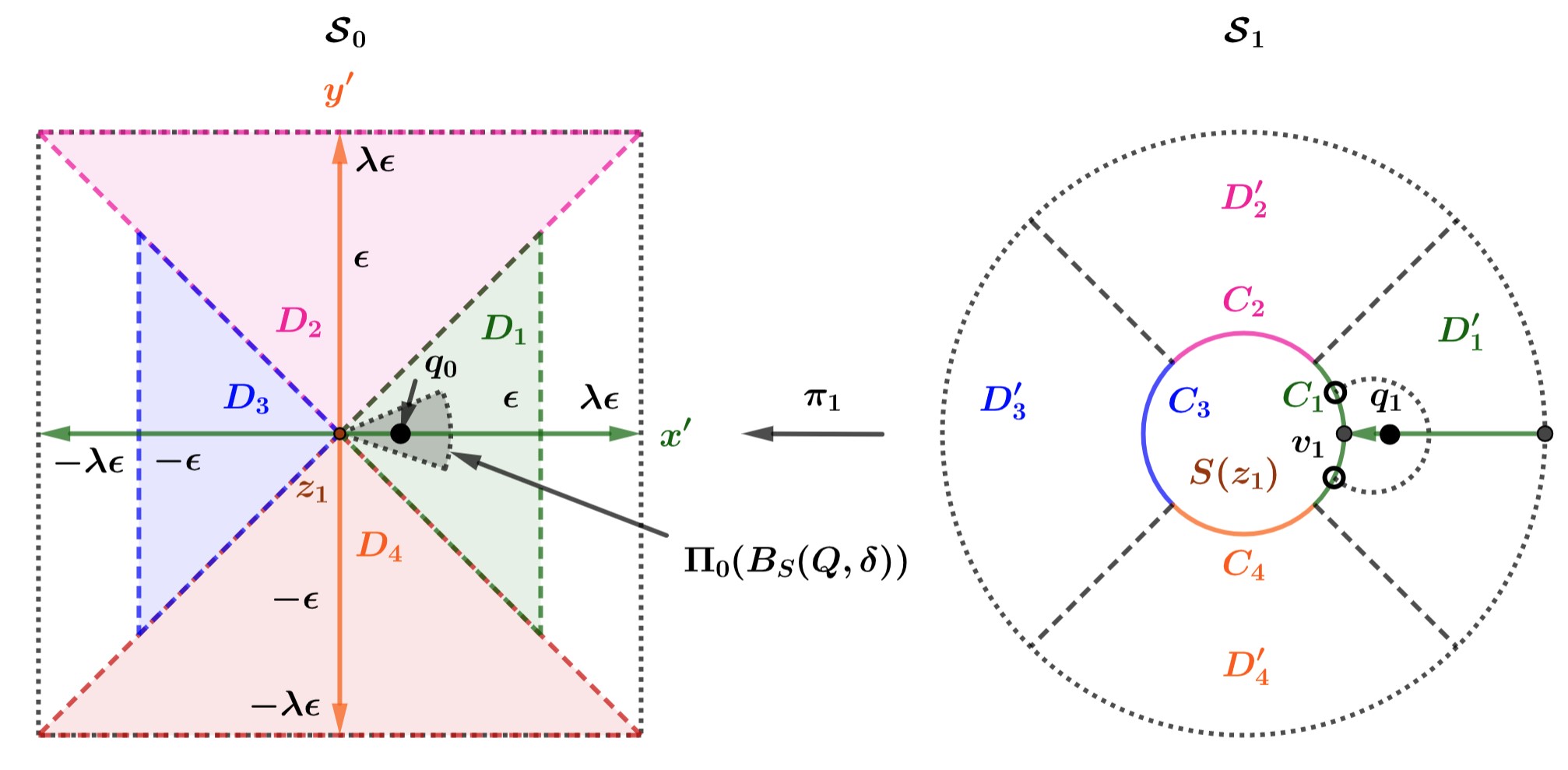}
    \caption{$\Pi_0(B_S(Q,\delta)) \subseteq D_1$ for small $\delta > 0$ for any points on the contracting lines\,.}
\end{figure}\\
%----------------------------------------
So for any small $\delta > 0$, $\Pi_0(B_S(Q, \delta)) \subseteq D_1$ for all $Q=(q_i)$ with $q_0 = z_1 + (x_Q, 0), 0 < x_Q < \epsilon - \delta.$
\end{proof}
%--------------------------------------
%------------------------------------------
\begin{theorem}\label{thm:502}
The system $(S^*, H^*)$ does not have the specification.
\end{theorem}
\begin{proof}
Allow us to abuse notation and suppose that $G := (H^*)^{n_1}$ has the specification.\\
Let $\delta > 0$ be sufficiently small (according to Lemma \ref{lem:502}).\\ Let $N = N(\delta)$ be the parameter for the specification of $G$ with respect to $\delta$. \\
Let $Q(i) = (q_j(i)) \in S^*$ where $q_0(i) = z_1 + ((-1)^{i+1}x_{Q(i)},0)$ with $0 < |x_{Q(i)}| < \epsilon - \delta$ for $i = 1, 2.$ \\
Let $0 = i_1 < j_1 = N+1$ and $i_2 = j_2 = 2N+1.$ Then there exists $u = (u_i) \in S^*$ such that for each $1 \leq k \leq 2$, $$G^{l_k}(u) \in B_S(G^{l_k}(Q(k), \delta))$$ for $l_k \in [i_k, j_k].$ This yields that, $H_0^{l}(u_0) \in D_1$ for all $l = 0 , 1, \ldots, N+1.$ \\So there exists $M \geq N+1$ such that $$H_0^l(u_0) \in D_1 \cup D_2 \cup D_4$$ for all $l = 0, 1, \ldots, 2M$. This is a contradiction because $2M > 2N+1.$
\end{proof}
%-------------------------------------------
\begin{remark}
Lemma \ref{lem:502} together with a slight modification of the proof of Theorem \ref{thm:501} yield that $(S^*, H^*)$ does not have the approximate product property  as well.
\end{remark} 
%%%%%%%%%%%%%%%%%%%%%%%%%%%%%%%%%%%%%%%%%%%%%%%%%%%%
\bibliographystyle{plain}
\bibliography{ArXiv} %% This is your file for bibtex
\nocite{*}
\end{document}